\documentclass[12pt,reqno]{amsart}
\usepackage[utf8]{inputenc}
\usepackage{amsmath,amsthm,amssymb,url,mathdots,verbatim}
\usepackage[margin = 2.5cm]{geometry}
\usepackage{xcolor}
\definecolor{evencell}{gray}{0.8}
\definecolor{oddcell}{gray}{1}
\definecolor{s-cell}{gray}{0}
\definecolor{gcell}{named}{green}
\definecolor{ycell}{named}{yellow}
\definecolor{ccell}{named}{green}
\definecolor{leftcell}{named}{violet}
\definecolor{middlecell}{named}{orange}
\definecolor{rightcell}{named}{teal}

\def\ec{*(evencell) \empty}
\def\oc{*(oddcell) \empty}
\def\SC{*(s-cell) \empty}
\def\gc{*(gcell) G}
\def\yc{*(ycell) Y}

\def\lc{*(leftcell) \empty}
\def\mc{*(middlecell) \empty}
\def\rc{*(rightcell) \empty}

\usepackage{tikz}
\usepackage[centerboxes]{ytableau}
\usepackage[colorlinks=true]{hyperref}
\usepackage{MnSymbol}
\usepackage{graphicx}
\usepackage{enumitem}
\usepackage[normalem]{ulem}
\usepackage[capitalise]{cleveref}
\crefname{figure}{Figure}{Figures}
\usepackage{comment}

\numberwithin{equation}{section}

\newcommand{\FindStat}[1]{%
  \ifx&#1&%
  \url{www.findstat.org}
  \else \url{www.findstat.org/#1}
  \fi}%

\newtheorem{theorem}{Theorem}
\newtheorem{proposition}[theorem]{Proposition}
\newtheorem{corollary}[theorem]{Corollary}
\newtheorem{lemma}[theorem]{Lemma}

\newtheorem{construction}{Construction}
\theoremstyle{definition}

\newtheorem{example}[theorem]{Example}
\newtheorem{remark}[theorem]{Remark}

\newcommand{\Dfn}[1]{\emph{\color{blue}#1}} 

\DeclareMathOperator{\wmaj}{wmaj}

\DeclareMathOperator{\gout}{green}
\DeclareMathOperator{\yout}{yellow}
\DeclareMathOperator{\leg}{leg}
\DeclareMathOperator{\arm}{arm}
\DeclareMathOperator{\hook}{hook}

\newcommand{\rem}{{\boldsymbol\rho}}
\newcommand{\pos}{{\boldsymbol\gamma}}
\newcommand{\posi}{{\gamma}}
\newcommand{\lab}[1]{\hbox{${\lambda}\!\downarrow_{#1}$}}
\newcommand{\lau}[1]{\hbox{${\lambda}\!\uparrow_{#1}$}}

\begin{document}

\title{A generalization of conjugation of integer partitions}

\author[Albion]{Seamus Albion$^\dagger$}
\address{$^\dagger$Fakult\"at f\"ur Mathematik, Universit\"at Wien,
Oskar-Morgenstern-Platz~1, Vienna, Austria\newline\tt
\url{http://www.mat.univie.ac.at/~aseamus}\newline
\url{http://www.mat.univie.ac.at/~teisenko}\newline
\url{http://www.mat.univie.ac.at/~ifischer}\newline
\url{http://www.mat.univie.ac.at/~kratt}\newline
\vskip-1.1cm}
\author[Eisenk\"olbl]{Theresia Eisenk\"olbl$^\dagger$}
\author[Fischer]{Ilse Fischer$^\dagger$}
\author[Gangl]{Moritz Gangl$^\dagger$}
\author[H\"ongesberg]{Hans H\"ongesberg$^\ddagger$}
\address{$^\ddagger$Faculty of Mathematics and Physics, University of
Ljubljana, Jadranska 21, Ljubljana, Slovenia\newline\tt
\url{http://hoengesberg.com}\vskip-.1cm}
\author[Krattenthaler]{Christian Krattenthaler$^\dagger$}
\author[Rubey]{Martin Rubey$^*$}
\address{$^*$Fakult\"{a}t f\"{u}r Mathematik und Geoinformation, TU Wien, Vienna, Austria\newline\tt
\url{https://www.dmg.tuwien.ac.at/rubey/}\newline
}
\thanks{S.A., I.F. and C.K. acknowledge support from the Austrian Science Fund (FWF) grant 10.55776/F1002.  M.G.  and H.H. acknowledge support from the FWF grant 10.55776/P34931, and H.H. also acknowledges support from the FWF grant 10.55776/J4810. For open access purposes, the authors have applied a CC BY public copyright license to any author accepted manuscript version arising from this submission.}

\begin{abstract}
    We exhibit, for any positive integer {parameter} $s$, an
    involution on the set of integer partitions of $n$.  These
    involutions show the joint symmetry of the distributions of the
    following two statistics.  The first counts the number of parts of
    a partition 
divisible by $s$, whereas the second counts the number of cells in the
    Ferrers diagram of a partition 
whose leg length is zero and whose arm length has remainder $s-1$ when
    dividing by $s$.  In particular, for $s=1$ this involution is just
    conjugation. 
    Additionally, we provide explicit expressions for the bivariate
    generating functions. 

    Our primary motivation to construct these involutions is that we
    know only of two other ``natural''  bijections on integer
    partitions of a given size, one of which is the Glaisher--Franklin
    bijection sending the set of parts divisible by $s$, each divided
    by $s$, to the set of parts occurring at least $s$ times.
\end{abstract}



\maketitle

\section{Introduction}

Integer partitions are possibly one of the most important families of
objects in combinatorics.  However, it seems that
we do not know of very many bijections on the set of integer
partitions of a given size --- although a large variety of bijections
between sets of partitions with certain properties can be found in
the literature, as witnessed by Pak {in his} survey~\cite{MR2267263}.

Apart from conjugation of the Ferrers diagram, a well-known family of
bijections is due to Glaisher and Franklin,
see~\cite[Sec.~3.3]{MR2267263}, \cite{zbMATH02703427}.\footnote{The
  case $s=2$ is \FindStat{Mp00312}.} For a given positive integer $s$,
it sends the set of parts divisible by $s$, each divided by $s$, to
the set of parts occurring at least $s$ times.

The other family of bijections we know of is due to Loehr and
Warrington~\cite{MR2475023}.  For each rational number $x$, they
describe an involution that interchanges two statistics $h^+_x$ and
$h^-_x$, which count the number of cells in the Ferrers diagram of a
partition satisfying certain constraints on the ratio of arm and leg
length.  These involutions can be combined, for example, to provide a
bijection sending the diagonal inversion number to the length of a
partition.\footnote{\FindStat{Mp00322}}

The purpose of this article is to present a family of involutions
on the set of partitions of a given integer that interchange two statistics
$r_s$ and~$c_s$ (to be defined in the next section), where $s$ is a
positive integer. For $s=1$ we recover the operation of conjugation.

To give an outline,
in the next section we recall standard notation and give
definitions relevant for our considerations. In particular,
there we introduce the announced statistics~$r_s$ and~$c_s$.
Subsequently we present our results. \cref{main} says that there
is an involution on partitions of~$n$ that interchanges the
statistics~$r_s$ and~$c_s$. The theorem is actually much finer as it
leaves the sequence of the non-zero remainders after division of the
parts of the partition by~$s$ invariant. Our second main result is presented in
\cref{coefficient}. It provides an explicit expression for the
generating function $\sum_\lambda q^{|\lambda|}$,
where the sum is over all partitions with
$(r_s(\lambda),c_s(\lambda))=(r,c)$ and a fixed sequence of
non-zero remainders, with $|\lambda|$ denoting the sum of parts of~$\lambda$.
The symmetry in~$r$ and~$c$ is evident from a slight modification of the expression in \cref{coefficient}, see
\cref{rem:1a}. 

\cref{sec:3,sec:4} are devoted to the construction of 
the involution of \cref{main}. The involution is built up step by step. It is
particularly simple if all parts of the partition are divisible
by~$s$, see \cref{const:1} in \cref{sec:empty}. The next case
that we consider is the case of strictly increasing remainder
sequences. The simple idea of \cref{const:1} is enhanced
by the operation of ``removal of remainders"
(see \cref{crucial}) and the concept of the
``remainder diagram".
The result is the more general involution
in \cref{strict} presented in \cref{sec:strict}.
In \cref{sec:4} it is argued that the general case can be reduced to
the case of strictly increasing remainder sequences, see \cref{reduce}.
The resulting complete description of our involution, proving
\cref{main}, is finally summarized in \cref{general}.

Along the way to this involution, we derive in parallel generating
function results, see \cref{zero,inc,gen}, and in particular
\cref{theogen}. (In fact, several ingredients to the involution are
inspired by generating function calculations.)
We complete the proof of \cref{coefficient} in
\cref{sec:5} by simplifying the expression from \cref{theogen}.
We offer actually two proofs of \cref{coefficient}: one
uses a combination of combinatorial arguments and $q$-series
identities, the other is purely combinatorial.
Finally, \cref{rem:1a} provides an alternative way to write the
expression for the generating function of \cref{coefficient},
which reveals the symmetry in $r$ and $c$.

The family of involutions we present here, depending on a positive
integer $s$, was discovered by an automated search for
equidistributed statistics on integer partitions in \FindStat{} such
that there is no accompanying bijection in the database.\footnote{The
case $s=2$ is now \FindStat{Mp00321}.
{\tt SageMath} code implementing the bijection for general $s$ can also be found there.}

\section{Definitions and Results}
\label{sec:2}
A \Dfn{partition} $\lambda$ of a positive integer $n$ is a weakly decreasing sequence of positive integers that add up to $n$. We write $\lambda \vdash n$ and $n$ is also referred to as the \Dfn{size} of~$\lambda$, denoted by  $|\lambda|$.
The number of parts is the \Dfn{length} of the partition, denoted by  $\ell(\lambda)$.
The \Dfn{Ferrers diagram} of $\lambda=(\lambda_1,\dots,\lambda_\ell)$
is the arrangement of left-justified unit boxes, called \Dfn{cells}, with $\lambda_i$ cells in row $i$.
In the following, we often identify the Ferrers diagram with the partition. We use the English convention and matrix coordinates to locate cells in the Ferrers diagram.  By $\lambda'$ we denote the \Dfn{conjugate partition} of~$\lambda$, which is obtained by reflecting the Ferrers diagram about the main diagonal. The \Dfn{leg length} $\leg(z)$ of a cell $z$ in the partition is the number of cells in the same column strictly below
the cell, while the \Dfn{arm length} $\arm(z)$ of a cell is the number of cells in the same row strictly to the right
of the cell. The Ferrers diagrams of
$\lambda=(6,4,4,1)$ and of $\lambda'=(4,3,3,3,1,1)$ are
shown in \cref{fig:1}.

\begin{figure}[ht]
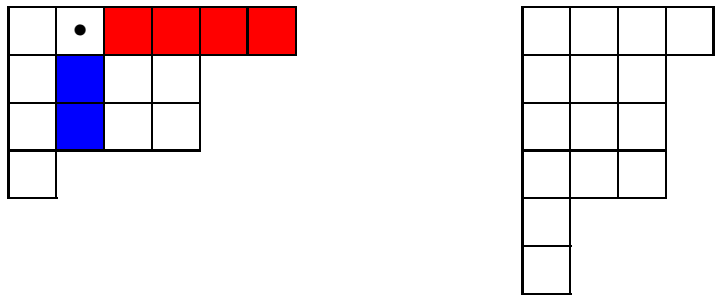

$$
\begin{ytableau}
\empty & \bullet & *(red) \empty  & *(red) \empty & *(red) \empty & *(red) \empty \\
\empty & *(blue) \empty & \empty & \empty \\
\empty & *(blue)  & \empty  & \empty \\
\empty  \\
\end{ytableau}
\hspace{3cm}
\begin{ytableau}
\empty & \empty & \empty & \empty \\
\empty & \empty & \empty \\
\empty & \empty & \empty \\
\empty & \empty & \empty \\
\empty \\
\empty \\
\end{ytableau}
$$
\caption{The Ferrers diagrams of
$\lambda=(6,4,4,1)$ and $\lambda'=(4,3,3,3,1,1)$}
\label{fig:1}
\end{figure}

The cells that contribute to the leg and arm lengths of the cell $(1,2)$ of the Ferrers diagram
of~$\lambda$ are indicated in blue and red, respectively, where the cell in the $i$-th row and $j$-th column is referred to as $(i,j)$.

\medskip

Throughout, we fix a positive integer $s$.  We define the following two statistics on partitions that depend on $s$. We let%
\footnote{We use the letter ``$r$" in $r_s$ and the letter ``$c$" in $c_s$
since, clearly, the first statistics is associated with the rows of the
Ferrers diagram of the partition, and since we think of the latter statistics
to be associated with the columns of the Ferrers diagram.}
\begin{align*}
r_s(\lambda) &= \# \text{ of parts of $\lambda$ divisible by $s$}, \\
c_s(\lambda) &= \# \text{ of cells $z$ in $\lambda$ such that $\leg(z)$ is zero and $\arm(z)+1$ is divisible by $s$.}
\end{align*}
A cell that contributes to $c_s(\lambda)$
is called \Dfn{$s$-cell}.  For example, given $\lambda=(6,4,4,1)$, the $2$-cells are $(1,5)$ and $(3,3)$ and we have
$r_2(\lambda)=3$ and $c_2(\lambda)=2$.

Our main goal is to show that the polynomial
$$\sum_{\lambda \vdash n} R^{r_s(\lambda)} C^{c_s(\lambda)}$$
is symmetric in $R$ and $C$ by constructing an involution on partitions
of $n$ that interchanges the statistics $r_s$ and $c_s$.

We will actually show a {vast} refinement of this statement.
The \Dfn{remainder sequence} of a partition~$\lambda$ modulo~$s$ is the sequence $\rem_s(\lambda)=(\rho_1,\dots, \rho_m)$ of non-zero remainders of the parts of~$\lambda$ when dividing by $s$ and reading $\lambda$ from left to right.
For example, given $\lambda=(12,9,5,4,4,3,2)$, we have $\rem_4(\lambda)=(1,1,3,2)$.  Our involution will fix the remainder sequence of the partition.
As a consequence, we obtain our first main theorem.

\begin{theorem}
\label{main}
Let $s$ and $n$ be positive integers, and let $\rem$ be a vector of integers between $1$ and~$s-1$. Furthermore, let $r$ and $c$ {be} non-negative integers. Then the number of partitions~$\lambda$ of $n$ with $\rem_s(\lambda)=\rem$ and
$(r_s(\lambda),c_s(\lambda))=(r,c)$ is equal to the number of partitions~$\lambda$ of~$n$ with $\rem_s(\lambda)=\rem$ and $(r_s(\lambda),c_s(\lambda))=(c,r)$.
\end{theorem}

\begin{example} \label{ex:1}
We choose $s=3$. There are exactly $5$~partitions~$\lambda$ of $37$ with
remainder sequence $(2,1,1,2,1)$, $r_3(\lambda)=2$ and
$c_3(\lambda)=3$, namely
$(15,6,5,4,4,2,1)$,
$(15,8,4,4,3,2,1)$,
$(14,10,4,3,3,2,1)$,
$(17,6,4,4,3,2,1)$, and
$(14,7,7,3,3,2,1)$. Their Ferrers diagrams are shown in \cref{fig:2}.
There, the $3$-cells are the black cells. Furthermore,
blocks of three cells are either white or shaded in order to
facilitate the identification of the row lengths that are divisible
by $s=3$.

\begin{figure}[ht]
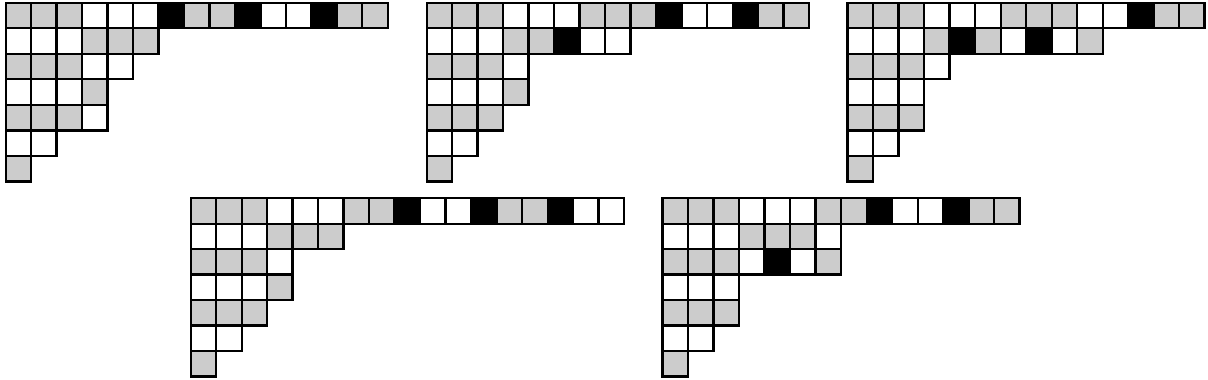

$$
\ytableausetup{smalltableaux}
\begin{ytableau}
\ec & \ec & \ec & \oc & \oc & \oc & \SC & \ec & \ec &
\SC & \oc & \oc & \SC & \ec & \ec \\
\oc & \oc & \oc & \ec & \ec & \ec \\
\ec & \ec & \ec & \oc & \oc \\
\oc & \oc & \oc & \ec \\
\ec & \ec & \ec & \oc \\
\oc & \oc \\
\ec \\
\end{ytableau}
\hspace{.5cm}
\begin{ytableau}
\ec & \ec & \ec & \oc & \oc & \oc & \ec & \ec & \ec &
\SC & \oc & \oc & \SC & \ec & \ec \\
\oc & \oc & \oc & \ec & \ec & \SC & \oc & \oc \\
\ec & \ec & \ec & \oc \\
\oc & \oc & \oc & \ec \\
\ec & \ec & \ec \\
\oc & \oc \\
\ec \\
\end{ytableau}
\hspace{.5cm}
\begin{ytableau}
\ec & \ec & \ec & \oc & \oc & \oc & \ec & \ec & \ec &
\oc & \oc & \SC & \ec & \ec \\
\oc & \oc & \oc & \ec & \SC & \ec & \oc & \SC & \oc & \ec \\
\ec & \ec & \ec & \oc \\
\oc & \oc & \oc \\
\ec & \ec & \ec \\
\oc & \oc \\
\ec \\
\end{ytableau}
$$
$$
\ytableausetup{smalltableaux}
\begin{ytableau}
\ec & \ec & \ec & \oc & \oc & \oc & \ec & \ec & \SC &
\oc & \oc & \SC & \ec & \ec & \SC & \oc & \oc \\
\oc & \oc & \oc & \ec & \ec & \ec \\
\ec & \ec & \ec & \oc \\
\oc & \oc & \oc & \ec \\
\ec & \ec & \ec \\
\oc & \oc \\
\ec \\
\end{ytableau}
\hspace{.5cm}
\begin{ytableau}
\ec & \ec & \ec & \oc & \oc & \oc & \ec & \ec & \SC &
\oc & \oc & \SC & \ec & \ec \\
\oc & \oc & \oc & \ec & \ec & \ec & \oc \\
\ec & \ec & \ec & \oc & \SC & \oc & \ec \\
\oc & \oc & \oc \\
\ec & \ec & \ec \\
\oc & \oc \\
\ec \\
\end{ytableau}
$$
\caption{The partitions $\lambda$ of 37 with remainder sequence
$(2,1,1,2,1)$, $r_3(\lambda)=2$
and $c_3(\lambda)=3$}
\label{fig:2}
\end{figure}

On the other hand, there are exactly $5$~partitions~$\lambda$ of $37$ with
remainder sequence\break $(2,1,1,2,1)$, $r_3(\lambda)=3$ and
$c_3(\lambda)=2$, namely
$(11,10,4,3,3,3,2,1)$,
$(12,8,4,4,3,3,2,1)$,
$(12,6,5,4,4,3,2,1)$,
$(11,7,7,3,3,3,2,1)$, and
$(14,6,4,4,3,3,2,1)$.
Their Ferrers diagrams are shown in \cref{fig:3}.
The shadings in the figure have the same meaning as in \cref{fig:2}.

Note that the fact that all partitions in \cref{fig:2,fig:3} have the same length, respectively, is a direct consequence of fixing the remainder sequence $\rem$ and the statistic $r_s$.

\begin{figure}[ht]
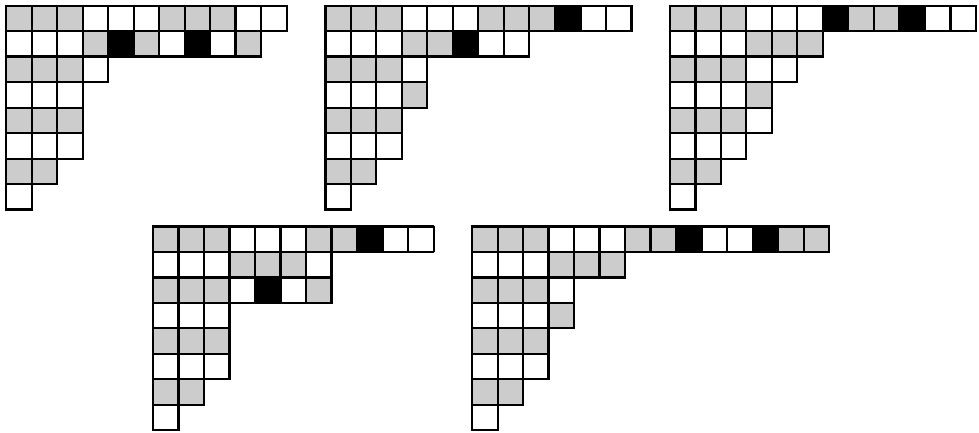

$$
\ytableausetup{smalltableaux}
\begin{ytableau}
\ec & \ec & \ec & \oc & \oc & \oc & \ec & \ec & \ec &
\oc & \oc  \\
\oc & \oc & \oc & \ec & \SC & \ec & \oc & \SC & \oc &
\ec  \\
\ec & \ec & \ec & \oc \\
\oc & \oc & \oc \\
\ec & \ec & \ec \\
\oc & \oc & \oc \\
\ec & \ec \\
\oc \\
\end{ytableau}
\hspace{.5cm}
\begin{ytableau}
\ec & \ec & \ec & \oc & \oc & \oc & \ec & \ec & \ec &
\SC & \oc & \oc \\
\oc & \oc & \oc & \ec & \ec & \SC & \oc & \oc \\
\ec & \ec & \ec & \oc \\
\oc & \oc & \oc & \ec \\
\ec & \ec & \ec \\
\oc & \oc & \oc \\
\ec & \ec \\
\oc \\
\end{ytableau}
\hspace{.5cm}
\begin{ytableau}
\ec & \ec & \ec & \oc & \oc & \oc & \SC & \ec & \ec &
\SC & \oc & \oc \\
\oc & \oc & \oc & \ec & \ec & \ec \\
\ec & \ec & \ec & \oc & \oc \\
\oc & \oc & \oc & \ec \\
\ec & \ec & \ec & \oc \\
\oc & \oc & \oc \\
\ec & \ec \\
\oc \\
\end{ytableau}
$$
$$
\ytableausetup{smalltableaux}
\begin{ytableau}
\ec & \ec & \ec & \oc & \oc & \oc & \ec & \ec & \SC &
\oc & \oc  \\
\oc & \oc & \oc & \ec & \ec & \ec & \oc \\
\ec & \ec & \ec & \oc & \SC & \oc & \ec \\
\oc & \oc & \oc \\
\ec & \ec & \ec \\
\oc & \oc & \oc \\
\ec & \ec \\
\oc \\
\end{ytableau}
\hspace{.5cm}
\begin{ytableau}
\ec & \ec & \ec & \oc & \oc & \oc & \ec & \ec & \SC &
\oc & \oc & \SC & \ec & \ec \\
\oc & \oc & \oc & \ec & \ec & \ec \\
\ec & \ec & \ec & \oc \\
\oc & \oc & \oc & \ec \\
\ec & \ec & \ec \\
\oc & \oc & \oc \\
\ec & \ec \\
\oc \\
\end{ytableau}
$$
\caption{The partitions $\lambda$ of 37 with remainder sequence
$(2,1,1,2,1)$, $r_3(\lambda)=3$
and $c_3(\lambda)=2$}
\label{fig:3}
\end{figure}
\end{example}
\ytableausetup{nosmalltableaux}

Apart from a bijective proof we also present a proof by computation. Both proofs imply the following result.
{In order to state it, recall that} the \Dfn{$q$-binomial coefficient} is defined as
$$
\begin{bmatrix}n \\ k \end{bmatrix}_q = \frac{[n]_q!}{[k]_q!\, [n-k]_q!}
$$
with
$$
[n]_q!=\prod_{i=1}^{n} (1+q+\dots+q^{i-1}) = \prod_{i=1}^n \frac{1-q^i}{1-q}.
$$
We extend the notion of size to finite sequences so that for $\rem=(\rho_1,\dots,\rho_m)$ we have $|\rem|=\rho_1+\dots+\rho_m$. We say that $\rem$ has a \Dfn{weak descent} at position $j$ if $\rho_j\geq\rho_{j+1}$.  Finally, the \Dfn{weak major index} of $\rem$ is the sum of the positions of its weak descents, that is
$$
\wmaj(\rem)= \sum_{j:\rho_j \ge \rho_{j+1}} j.
$$
This is a special case of the so called ``graphical major indices" introduced by Foata and Zeilberger \cite{MR1418752} and further investigated by Clarke and Foata \cite{MR1279073, MR1330539, MR1337139, MR1399503} as well as Foata and 
one of the authors \cite{MR1399758}. Using the language of 
the articles by Clarke and Foata, the ``weak major index" is the major index defined solely on ``large" letters.

Our announced generating function result is the following.

\begin{theorem}
\label{coefficient}
Let $s$ be a positive integer, $\rem$ be a vector of integers between $1$ and~$s-1$ of length~$m$, and $r,c$ be non-negative integers. The generating function with respect to the
weight $q^{|\lambda|}$ of partitions~$\lambda$ with $\rem_s(\lambda)=\rem$ and
$(r_s(\lambda),c_s(\lambda))=(r,c)$ is
$$
q^{|\rem|}  Q^{-\wmaj(\rem)+\binom{m}{2}+r+c} \bigg(\begin{bmatrix} r+m-1 \\ m-1 \end{bmatrix}_Q\begin{bmatrix} r+c+m-2 \\ c\end{bmatrix}_Q
    +Q^{m-1}\begin{bmatrix} r+m \\ m \end{bmatrix}_Q
\begin{bmatrix} r+c+m-2 \\ c-1 \end{bmatrix}_Q \bigg),
$$
where $Q=q^s$.
\end{theorem}

A surprising feature of the formula is that the dependence on the
remainder sequence~$\rem$ is only in the exponent of~$q$ in front of the
expression. This ``almost-independence" from~$\rem$ is explained by
the bijection of \cref{reduce}.

In the following corollary, we provide an alternative way to write the expression in Theorem~\ref{coefficient} from which the symmetry in $r$ and $c$ expressed in \cref{main}
is obvious.

\begin{corollary}\label{rem:1a}
The generating function in \cref{coefficient} is equal to
$$
q^{|\rem|} Q^{-\wmaj(\rem)+\binom{m}{2}+r+c}
\bigg(\frac {[r+c+m-1]_Q!} {[r]_Q!\,[c]_Q!\,[m-1]_Q!}
+Q^{m-1}\frac {[r+c+m-2]_Q!} {[r-1]_Q!\,[c-1]_Q!\,[m]_Q!}
\bigg).
$$
\end{corollary}

\begin{example}
If we choose $s=3$, $m=5$, $\rem=(2,1,1,2,1)$, and $(r,c)=(2,3)$
(respectively $(r,c)=(3,2)$) in the formula of \cref{coefficient},
then we obtain
\begin{multline*}
q^{7}  q^{3\cdot(-7+10+2+3)} \bigg(\begin{bmatrix} 6 \\ 4 \end{bmatrix}_{q^3}\begin{bmatrix} 8 \\ 3\end{bmatrix}_{q^3}
    +q^{3\cdot4}\begin{bmatrix} 7 \\ 5 \end{bmatrix}_{q^3}
\begin{bmatrix} 8 \\ 2 \end{bmatrix}_{q^3} \bigg)\\
=q^{31}\left(1+2q^3+5q^6+9q^9+17q^{12}+\dots+16 q^{66}+
9q^{69}+5q^{72}+2q^{75}+q^{78}\right).
\end{multline*}
In particular, the coefficient of $q^{37}$ in this polynomial
equals~5, corresponding to the five partitions for each of
$(r,c)=(2,3)$ and $(r,c)=(3,2)$ in \cref{ex:1}.
\end{example}

\begin{remark} \label{rem:BF}
	It is worth pointing out that the statistic $c_s$ occurred earlier in an algebraic context as a special case of a more general statistic.  Given integers $\alpha\geq 1$ and $\beta\geq 0$ and a partition $\lambda$, define the set 
	\[
	\mathrm{bf}_{\alpha,\beta}(\lambda)
	\coloneq \{z\in\lambda : 
	\text{$\alpha \leg(z)=\beta(\arm(z)+1)$ and $\hook(z)\equiv 0 \mod{(\alpha+\beta)}$}\},
	\]
	where $\hook(z)=\arm(z) + \leg(z) + 1$ is the usual \Dfn{hook length}.  Furthermore let
	$\mathrm{BF}_{\alpha,\beta}(\lambda)\coloneq|\mathrm{bf}_{\alpha,\beta}(\lambda)|$.  Then $c_s(\lambda) = \mathrm{BF}_{s,0}(\lambda)$.
	
	This statistic was first defined by Buryak and Feigin \cite{BF13}
	in the case that $\alpha$ and $\beta$ are coprime, and extended to arbitary
	$\alpha$ and $\beta$ in joint work with Nakajima \cite{BFN15}. Using the  standard notation for \Dfn{$q$-shifted factorials}
	\[
	(a;q)_n \coloneq \prod_{i=0}^{n-1}(1-a q^i) \quad\text{and}\quad (a;q)_\infty \coloneq \prod_{i=0}^{\infty}(1-a q^i),
	\]
	they compute the generating function 
	\[
	\sum_{\lambda}t^{\mathrm{BF}_{\alpha,\beta}(\lambda)}q^{|\lambda|}
	=\frac{(q^{\alpha+\beta};q^{\alpha+\beta})_\infty}
	{(q;q)_{\infty}(tq^{\alpha+\beta};q^{\alpha+\beta})_\infty}
	\]
	of all integer partitions~$\lambda$, which shows that $\mathrm{BF}_{\alpha,\beta}(\lambda)$ and 
	$\mathrm{BF}_{\alpha+\beta,0}(\lambda)$ are equidistributed over partitions of $n$,
	since the right-hand side depends only on $\alpha+\beta$.
	Note that the right-hand side is also the generating function of partitions where $t$ counts the number of parts divisible by $\alpha+\beta$.
	
	The main result of both \cite{BF13} and \cite{BFN15} is the computation of the Poincar\'e polynomial of certain Hilbert schemes, called quasihomogeneous Hilbert schemes. These polynomials may be expressed as generating functions
	for the statistic $\mathrm{BF}_{\alpha,\beta}(\lambda)$ summed over all partitions $\lambda$ of $n$.
	See also \cite{Vid23} and \cite{WW20} for further explanation.

	In \cite{Vid23}, Vidalis relates the statistic 
	$\mathrm{BF}_{\alpha,\beta}(\lambda)$ and the statistics 
	$h_x^+(\lambda)$ and $h_x^-(\lambda)$ of Loehr and Warrington
	through a generalisation, denoted by $h_{x,s}^\pm(\lambda)$.
	It is shown that
	the two statistics $h_{x,s}^+(\lambda)$ and $h_{x,s}^-(\lambda)$ have symmetric
	joint distribution when restricted to partitions with a fixed $s$-core. Moreover, if $x$ is a rational number of the form
	$\alpha/\beta$ with $\alpha+\beta$ dividing~$s$, then $h_{x,s}^+(\lambda)$
	has the same distribution as $c_s$.
	The statistics $r_s$ and $c_s$, however, do not have a symmetric joint distribution when restricted to partitions with a fixed $s$-core as the following counterexample shows: A partition is called an \Dfn{$s$-core} if none of the hook lengths are divisible by $s$. Among the partitions of $6$, there is exactly one $2$-core, namely $\lambda=(3,2,1)$. For this partition we have $r_2(\lambda)=1$ and $c_2(\lambda)=0$.
\end{remark}

\section{Some special cases}
\label{sec:3}

The purpose of this section is to define the involution of \cref{main}
in two simpler cases: first for the case where all parts of the
partitions are divisible by~$s$ (see \cref{const:1}), and then for the
more general case where the non-zero remainders that the parts leave after
division by~$s$ are in increasing order (see \cref{strict}).
Moreover, we provide the necessary auxiliary results that imply that the
constructed mappings are indeed involutions and have the desired
properties in relation to the statistics~$r_s$ and~$c_s$.
Finally, working towards the proof of \cref{coefficient}, we also
provide corresponding generating function results, the upshot
being \cref{inc}.

\subsection{Bijective proof for the case of the empty remainder sequences.}
\label{sec:empty}
In the special case where the remainder sequence of $\lambda=(\lambda_1,\dots,\lambda_\ell)$ modulo~$s$ is empty, each row of the Ferrers diagram can be
partitioned into segments of length~$s$. We shrink each of these segments to one cell, i.e., we consider the partition
$(\frac{\lambda_1}{s},\frac{\lambda_2}{s},\dots,\frac{\lambda_l}{s})$. Then $r_s(\lambda)$ is the number of rows and $c_s(\lambda)$ is the number of columns of the shrunk partition. In this case, conjugation of the shrunk diagram and 
subsequent expansion of each cell again into
a row segment of length~$s$ give the involution.

This involution is also the basis for the general case.  To describe it
formally, we introduce some notation.  On the one hand, let
$$
\lab{s} =\left( \lfloor \lambda_1/s \rfloor, \dots,  \lfloor \lambda_\ell/s \rfloor \right).
$$
We call $\lab{s}$ the \Dfn{$s$-reduction} of~$\lambda$. On the other hand, let
$$
\lau{s} = \left( s \cdot \lambda_1,\dots,  s \cdot \lambda_\ell \right)
$$
and call $\lau{s}$  the \Dfn{$s$-blow-up} of~$\lambda$.

We have $\lau{s}\! \downarrow_s = \lambda$
for every partition~$\lambda$.  We also have
$\lab{s}\! \uparrow_s = \lambda$
if and only if $\lambda$ has empty remainder sequence modulo~$s$. 
{The involution on partitions with empty remainder sequence can now be stated as follows.}

\begin{construction}[\sc Empty remainder sequence] \label{const:1}
Let $\lambda$ be a partition with empty remainder sequence. We define the mapping
  $$\lambda \mapsto [\lab{s}]' \!\uparrow_s.$$
\end{construction}
Our reasoning above demonstrates that \cref{const:1} is an involution on
partitions with empty remainder sequence modulo~$s$ that interchanges the statistics $r_s$ and $c_s$.

The discovery of the general involution, proving 
\cref{main}, was inspired by generating function considerations
that led to a proof of \cref{coefficient}, as indicated throughout the presentation.  The preceding construction thus corresponds to the statement of the following lemma.  
To this end, recall the definition of $q$-shifted factorials introduced earlier in \cref{rem:BF} as
\[
(a;q)_n \coloneq \prod_{i=0}^{n-1}(1-a q^i).
\]

\begin{lemma}
\label{zero}
The generating function with respect to the weight $R^{r_s(\lambda)} C^{c_s(\lambda)} q^{|\lambda|}$
of partitions~$\lambda$ with empty remainder sequence is given by
$$
1 + \sum_{k \ge 1} R^k \frac{C Q^{k}}{(C Q; Q)_k},
$$
where $Q=q^s$.
\end{lemma}

\begin{proof}
By shrinking row segments of length~$s$ as above, it suffices to compute the generating function of all partitions~$\lambda$ with respect to the weight
$$
R^{\# \text{ of rows of $\lambda$}} C^{\# \text{ of columns of $\lambda$}} q^{|\lambda|}
$$
and then replace $q$ by $Q=q^s$, which takes care of expanding the cells into row segments again.

The weight of the empty partition is $1$.  We show that the
generating function of non-empty partitions of length $k$ with respect to the above weight is
\[
R^k \frac{C q^{k}}{\prod_{i=1}^{k}(1-C q^{i})}.
\]
Indeed, the weight of the first column of $\lambda$, necessarily of
length $k$, is $R^k C q^k$, whereas $\frac{1}{1-Cq^i}$ is the
generating function of rectangular partitions with exactly $i$ rows.
\end{proof}

\subsection{A crucial operation: 
removal of the final non-zero remainder}
\label{crucial}

In the following, we also need to keep track of the positions of 
parts with a non-zero remainder modulo~$s$ in a partition~$\lambda$.  
We define the \Dfn{row position sequence} $\pos_s(\lambda)=(\posi_1,\dots,\posi_m)$ to be the sequence of indices $1 \le \posi_1 < \dots < \posi_m$ such that $\lambda_{\posi_j}$ has non-zero remainder after division by $s$.  Let
$
\Delta_s \lambda
$
be the partition we obtain by deleting the
last $\rho_m$ cells in the $\posi_m$-th row of the Ferrers diagram of~$\lambda$.
\begin{lemma}
  \label{delete}
  Let $\lambda$ be a partition with remainder sequence
  $\rem_s(\lambda) = (\rho_1,\dots,\rho_m)$ and row position sequence
  $\pos_s(\lambda) = (\posi_1,\dots,\posi_m)$, $m\geq 1$.  Then
  \[
  c_s(\Delta_s \lambda)
  =
  \begin{cases}
    c_s(\lambda), & \text{if $m=1$ and $\posi_1=1$, or $m >1$ and $\posi_{m-1}=\posi_{m}-1$ and $\rho_{m-1} \ge \rho_m$},\\
    c_s(\lambda)+1,
 & \text{otherwise}.
  \end{cases}
  \]
\end{lemma}

Before providing the proof, we illustrate the result with the help of an example: Setting $s=3$, we have $c_s(\Delta_s \lambda) = c_s(\lambda)+1$ for the
partition on the left in \cref{fig:4}, and  $c_s(\Delta_s \lambda) = c_s(\lambda)$ for the partition on the right.  The non-zero remainders are 
indicated in green.
\begin{figure}[ht]
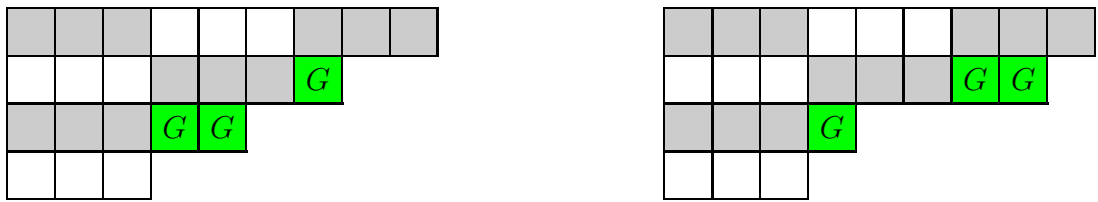

$$
\begin{ytableau}
\ec & \ec & \ec & \oc & \oc & \oc & \ec & \ec & \ec \\
\oc & \oc & \oc & \ec & \ec & \ec & \gc \\
\ec & \ec & \ec & \gc & \gc \\
\oc & \oc & \oc
\end{ytableau}
\hspace{3cm}
\begin{ytableau}
\ec & \ec & \ec & \oc & \oc & \oc & \ec & \ec & \ec \\
\oc & \oc & \oc & \ec & \ec & \ec & \gc & \gc \\
\ec & \ec & \ec & \gc \\
\oc & \oc & \oc
\end{ytableau}
$$
\caption{Example partitions for \cref{delete}}
\label{fig:4}
\end{figure}

\begin{proof}
First note that the case where $m=\gamma_1=1$ is immediate.
From now on we tacitly assume that we are not in this case.

Next we observe that
the deletion only has an effect on the $s$-cells in rows $\posi_m-1$ and $\posi_m$. However, since the lengths of the rows below row $\posi_m$ are all divisible by $s$, the number of $s$-cells in row $\posi_m$ does not change. If $\posi_{m-1}<\posi_m-1$, then there is one more $s$-cell in row $\posi_{m}-1$ of $\Delta_s \lambda$ than in the same row
of~$\lambda$. More formally, if $\posi_{m-1}<\posi_m-1$, then the length of row~$\pos_m-1$ is divisible by $s$. So assume that there are $k s+\rho_m$ cells in row~$\posi_m$ and $l s$ cells in row~$\pos_m-1$ of $\lambda$ for some integers $k$ and $l$ satisfying $k<l$. Then there are $\lfloor \frac{l s - k s -\rho_m}{s} \rfloor = l-k-1$ $s$-cells in row~$\posi_m-1$ since $0 < \frac{\rho_m}{s} < 1$. But after removing $\rho_m$ cells in row~$\posi_m$, we obtain $\lfloor \frac{l s - k s}{s} \rfloor = l-k$ $s$-cells in row~$\posi_m-1$.
This is still true if $\posi_{m-1}=\posi_m-1$ and $\rho_{m-1} < \rho_m$.
To see this, assume that, while row~$\posi_m$ has $k s+\rho_m$ cells, row~$\posi_{m-1}$ now has $l s + \rho_{m-1}$ cells. Since $\rho_{m-1} < \rho_{m}$, we still have $k<l$. Then there are $\lfloor \frac{l s + \rho_{m-1} - k s -\rho_m}{s} \rfloor = l-k-1$ $s$-cells in row~$\posi_{m-1}$ of $\lambda$ and $\lfloor \frac{l s + \rho_{m-1} - k s}{s} \rfloor = l-k$ $s$-cells in row~$\posi_{m-1}$ of $\Delta_s \lambda$.

However, if
$\posi_{m-1}=\posi_m-1$ and $\rho_{m-1} \ge \rho_m$, the number of $s$-cells does not change:
assume again that row~$\posi_m$ has $k s+\rho_m$ cells and row~$\posi_{m-1}$ has $l s + \rho_{m-1}$ cells for $k \le l$. Then we conclude that row~$\posi_{m-1}$ has $\lfloor \frac{l s + \rho_{m-1} - k s -\rho_m}{s} \rfloor = l-k+\lfloor \frac{\rho_{m-1} -\rho_m}{s} \rfloor$ $s$-cells before and $\lfloor \frac{l s + \rho_{m-1} - k s}{s} \rfloor =l-k+ \lfloor \frac{\rho_{m-1}}{s} \rfloor$ $s$-cells after removing $\rho_m$ in row~$\posi_m$. These expressions are equivalent provided that $0 < \rho_m \le \rho_{m-1} < s$.
\end{proof}

\subsection{Bijective proof for the case of strictly increasing
remainder sequences}
\label{sec:strict}
After having understood empty remainder sequences, the next easiest task is to accommodate strictly increasing remainder sequences. 
{The reason is that, in this case, the statistic $c_s$ increases by $1$ when successively removing the final non-zero remainders, i.e., 
$c_s(\Delta_s^i \lambda) = c_s(\Delta_s^{i-1} \lambda)+1$ for $i=1,\ldots,m$, except for the case when there is just one non-zero remainder left 
and it is the remainder of the first part of the partition.}

Let $\pos_s(\lambda)=(\posi_1,\dots,\posi_m)$ be the row position sequence of~$\lambda$.  The \Dfn{column position sequence}
$\pos'_s(\lambda)=(\posi'_1,\dots,\posi'_m)$ is the sequence $\left( \lceil \lambda_{\posi_1}/s \rceil, \dots,  \lceil \lambda_{\posi_m}/s \rceil \right)$.  Informally, these are the column indices corresponding to the removed remainders in~$\lab{s}$.
To give an example, let $s=4$ and let $\lambda$ be the partition $(4s+1, 4s, 3s+2, 3s, 2s, s+3)$.  Its Ferrers diagram is shown in \cref{fig:5} on the left, while its $s$-reduction is shown on the right (the bullets should be ignored at this point). In this example, we have
$\pos_s(\lambda)=(1,3,6)$ and $\pos'_s(\lambda)=(5,4,2)$, and the
remainder sequence is $\rem_s(\lambda)=(1,2,3)$ (corresponding to the
green cells in \cref{fig:5}).
\begin{figure}[ht]
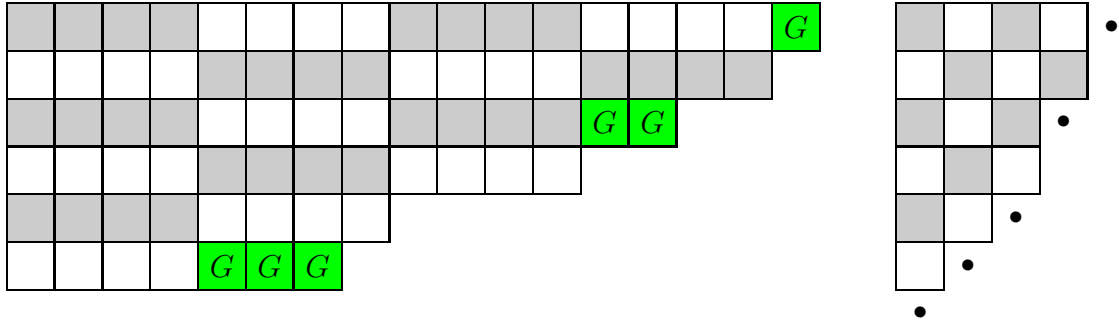

$$
\begin{ytableau}
\ec & \ec & \ec & \ec & \oc & \oc & \oc & \oc & \ec & \ec & \ec & \ec & \oc & \oc & \oc & \oc & \gc \\
\oc & \oc & \oc & \oc & \ec & \ec & \ec & \ec & \oc & \oc & \oc & \oc & \ec & \ec & \ec  & \ec \\
\ec & \ec & \ec & \ec & \oc & \oc & \oc & \oc & \ec & \ec & \ec & \ec & \gc & \gc \\
\oc & \oc & \oc & \oc & \ec & \ec & \ec & \ec & \oc & \oc & \oc & \oc \\
\ec & \ec & \ec & \ec & \oc & \oc & \oc & \oc  \\
\oc & \oc & \oc & \oc & \gc & \gc & \gc \\
\end{ytableau}
\hskip1cm
\begin{ytableau}
\ec & \oc & \ec & \oc & \none[\bullet] \\
\oc & \ec & \oc & \ec \\
\ec & \oc & \ec & \none[\bullet] \\
\oc & \ec & \oc \\
\ec & \oc & \none[\bullet] \\
\oc & \none[\bullet] \\
\none[\bullet] \\
\end{ytableau}
$$
\caption{A partition of 74 and its 4-reduction}
\label{fig:5}
\end{figure}

Given a partition~$\lambda$ with strictly increasing remainder sequence, the \Dfn{green cells} $\gout_s(\lambda)$ are defined as the cells $(\posi_1, \posi'_1),\dots,(\posi_m, \posi'_m)$.
In our running example of \cref{fig:5}, these are $(1,5)$, $(3,4)$, and $(6,2)$.

Recall that an \Dfn{outer corner} of a Ferrers diagram~$\lambda$ is a cell $z$ not contained in the diagram such that
the union $\lambda \cup z$ is a Ferrers diagram. For example, the outer
corners of the Ferrers diagram on the right of \cref{fig:5} are indicated
by black dots. Next we show that all green cells are outer corners of the $s$-reduction.
\begin{proposition}\label{outer}
  Let $\lambda$ be a partition with strictly increasing remainder sequence modulo~$s$.  Then the cells in $\gout_s(\lambda)$ are outer corners of~$\lab{s}$.
\end{proposition}
\begin{proof}
    We have $\lceil\lambda_i/s\rceil = \lfloor\lambda_i/s\rfloor + 1$ if and only if $\lambda_i$ is not divisible by $s$, so the cells in $\gout_s(\lambda)$ are indeed just outside of~$\lab{s}$. Since the remainder sequence is strictly increasing, the cells in $\gout_s(\lambda)$ have distinct column indices.
\end{proof}
The \Dfn{remainder diagram} $\nu^+_s(\lambda)$ is obtained from the Ferrers diagram of~$\lab{s}$ by adding the green cells,
as coloured cells.\footnote{The concept of the ``remainder diagram"
has some similarities with parts of the
Littlewood-like decomposition of partitions in~\cite[p.~12]{WW20},
although there does not seem to be a direct overlap.} 
We call $\lab{s}$ the \Dfn{interior} of $\nu^+_s(\lambda)$.
\cref{fig:6} displays the remainder diagram $\nu^+_s(\lambda)$ of the
partition~$\lambda$ from \cref{fig:5}. There, the green cells are
marked in green, while the remaining --- non-coloured --- cells
form the interior of $\nu^+_s(\lambda)$.
\begin{figure}[ht]
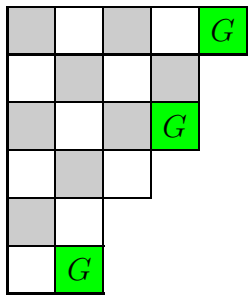

$$
\begin{ytableau}
\ec & \oc & \ec & \oc & \gc \\
\oc & \ec & \oc & \ec \\
\ec & \oc & \ec & \gc \\
\oc & \ec & \oc \\
\ec & \oc \\
\oc & \gc \\
\end{ytableau}
$$
\caption{The remainder diagram of the partition of \cref{fig:5}}
\label{fig:6}
\end{figure}

Next we show that
the statistics $r_s$ and $c_s$ are determined by the remainder diagram.
\begin{lemma}\label{lem:increasing}
  Let $\lambda$ be a partition with strictly increasing remainder sequence modulo~$s$ and remainder diagram $\nu^+_s(\lambda)$.
  Then
  \begin{align*}
    r_s(\lambda) &= \# \text{ of rows of } \nu^+_s(\lambda)-|\gout_s(\lambda)|, \\
    c_s(\lambda) &= \# \text{ of columns of } \nu^+_s(\lambda)-|\gout_s(\lambda)|.
  \end{align*}
\end{lemma}
\begin{proof}
    The first equation holds because the green cells correspond to the parts of~$\lambda$ which are not divisible by $s$.

    For a partition~$\lambda$ with empty remainder sequence, we have $|\gout_s(\lambda)|=0$ and $\nu^+_s(\lambda) = \lab{s}$, and $c_s(\lambda)$ equals the number of columns of~$\lab{s}$. If $m=1$ and $\posi_1=1$, then $c_s(\lambda)$ also equals the number of columns of~$\lab{s}$.  However, $\nu^+_s(\lambda)$ has precisely one more column than~$\lab{s}$.  Otherwise, by \Cref{delete}, each green cell of $\nu^+_s(\lambda)$ reduces the number of cells counted by $c_s(\lambda)$ by one.
\end{proof}

The conjugate of a remainder diagram is obtained in the same way as the conjugate of a Ferrers diagram, by reflecting about the main diagonal.  Thus, the green cells are at positions $(\posi'_1, \posi_1),\dots,(\posi'_m, \posi_m)$ of the conjugate remainder diagram.

Conjugating $\nu_s^+(\lambda)$, then expanding the cells of the interior again into row segments of $s$ cells and putting the remainders back into the green cells, in increasing order from top to bottom, we obtain the involution that swaps the two statistics in this special case.

To write down the bijection formally we need one further definition.  Let $\nu^+$ be a partition with $m$ coloured cells that are at the end of their respective rows, and let $\rem=(\rho_1,\dots,\rho_m)$ be a vector of integers between 1 and~$s-1$.  Then we define $\nu^+ \leftarrow_s \rem$ to be obtained from the $s$-blow-up of the interior of $\nu^+$
(that is, of the uncoloured cells) by adding $\rho_i$ cells to the rows corresponding to the coloured cells, in order.  

\begin{construction}[\sc Strictly increasing remainder sequence]
\label{strict}
Let $\lambda$ be a partition with strictly increasing remainder sequence $\rem=\rem_s(\lambda)$. We define the mapping
$$\lambda \mapsto \left(\left[\nu^+_s(\lambda)\right]' \leftarrow_s \rem\right).$$
\end{construction}
Our reasoning above demonstrates that \cref{strict} is an involution on
partitions with strictly increasing remainder sequence modulo~$s$
that interchanges $r_s$ and $c_s$.

\medskip
We now extend \Cref{zero} to the case of strictly increasing remainder sequences.
\begin{lemma}
\label{inc}
Let $\rem=(\rho_1,\dots, \rho_m)$ be a vector of integers between $1$ and~$s-1$ with strictly increasing coordinates. The generating function of partitions~$\lambda$
with $\rem_s(\lambda)=\rem$ with respect to the weight
$R^{r_s(\lambda)} C^{c_s(\lambda)} q^{|\lambda|}$ is given by
$$
q^{|\rem|} \sum_{1 \le \posi_1 < \posi_2 < \dots < \posi_m} Q^{|\pos|-m} \left( R^{\posi_m-m} + \sum_{k \ge 1}  \frac{C Q^{k}}{(C Q; Q)_k}
R^{\max(\posi_m-m,k-m)}
\right),
$$
where, as before, $Q=q^s$.
\end{lemma}

\begin{proof}
Let $\lambda$ be a partition with $\rem_s(\lambda)=\rem$ and $\pos_s(\lambda)=\pos$.
We modify $\lambda$ as follows: we delete the last $\rho_m$ cells in row $\posi_m$, i.e., we apply
$\Delta_s$ to~$\lambda$, and then delete $s$ cells in each row
strictly above row $\posi_m$. By 
\cref{delete}, this does not change the statistic $c_s$. These deletions are taken into account by the
terms $q^{\rho_m}$ and $Q^{\posi_m-1}$ in the generating function. We continue in this manner: we delete the last cells $\rho_{j}$ in row $\posi_{j}$ and delete $s$ cells in each row strictly above row $\posi_{j}$ for $j=m-1,m-2,\dots,1$. This does not change the statistic $c_s$ and, in total, the deletions are taken into account by the terms $q^{|\rem|}$ and $Q^{|\pos|-m}$.

We are left with a partition with empty remainder sequence. Suppose
$k$ is the length of this partition.
If the resulting partition is empty, corresponding to the case $k=0$, then the original partition~$\lambda$ has no rows below row~$\posi_m$. This implies that $\lambda$ has exactly $\posi_m$ parts, of which $\posi_m-m$ are divisible by $s$. This explains the term  $R^{\posi_m-m}$.
In the case $k \ge 1$, as can be seen in the proof of \Cref{zero}, the generating function of such partitions with respect to the weight $C^{c_s(\lambda)} q^{|\lambda|}$
is $\frac{C Q^{k}}{(C Q;Q)_k}$.
If $\posi_m > k$, then the original partition~$\lambda$ has $\posi_m$ parts and $\posi_m-m$ of them are divisible by $s$; compare with the case $k=0$. If $\posi_m \le k$, then it follows that $\lambda$ has $k$ parts. There are $\posi_m-m$ parts divisible by $s$ before position~$\posi_m$ and $k-\posi_m$ parts divisible by $s$ after position~$\posi_m$, which amount to a total of $k-m$ parts divisible by $s$. The assertion follows.
\end{proof}

\section{The general case}\label{sec:general}
\label{sec:4}

In this section we provide an algorithm, presented in
\cref{reduce}, that affords a
reduction of the general case to the case of strictly increasing
remainder sequences, the case that we had just discussed in
\cref{sec:strict}. This leads in particular to the completion
of the proof of \cref{main}, with the involution summarized in
\cref{general}. As already in the previous section, also here
we derive in parallel the corresponding generating function
results, culminating in \cref{theogen}, which constitutes the
basis for the proof of \cref{coefficient} in \cref{sec:5}.

Since it provides the inspiration for the
constructions to follow, we start from the generating
function side.
We show next how the observation from \cref{crucial} can be
used to generalize \Cref{inc} in a straightforward manner to the general case. In order to express the generating function, it is useful to define a $01$-sequence
${\mathbf d}(\rem,\pos)=(d_1,\dots,d_m)$ of length~$m$, which depends
on a vector $\rem$ of length~$m$ of integers between 1 and~$s-1$ and a strictly increasing sequence of positive integers
$\pos=(\posi_1,\dots,\posi_m)$ as follows; later on, $\posi_j$ will again be the row of the remainder $\rho_j$:  we set $d_j=1$ unless $j>1$, $\rho_{j-1} \ge \rho_j$ and $\posi_j=\posi_{j-1}+1$, in which case we set $d_j=0$. The motivation for this definition comes from the operation provided in \cref{crucial}. Note that $d_1=1$.

\begin{lemma}
\label{gen}
Let $\rem$ be a vector of integers between $1$ and~$s-1$ 
of length~$m$. The generating function
with respect to the
weight $R^{r_s(\lambda)} C^{c_s(\lambda)} q^{|\lambda|}$ of
partitions~$\lambda$ with remainder sequence~$\rem$ is given by
$$
q^{|\rem|} \sum_{1 \le \posi_1 < \posi_2 < \dots < \posi_m} Q^{{\mathbf d}(\rem,\pos) \cdot ({\pos - \mathbf 1})} \left( R^{\posi_m-m} +
\sum_{k \ge 1}  \frac{C Q^{k}}{(C Q; Q)_k}
 R^{\max(\posi_m-m,k-m)} \right),
$$
where $Q=q^s$ and\/
${\mathbf d}(\rem,\pos) \cdot ({\pos- \mathbf 1})$ denotes the standard inner
product of ${\mathbf d}(\rem,\pos)$ and
$(\pos- \mathbf 1) = (\posi_1-1,\dots,\posi_m-1)$.
\end{lemma}

\begin{proof}
The proof follows essentially the steps from the proof of \Cref{inc}, except for the following detail: when we delete the last $\rho_j$ cells in row $\posi_j$ then we delete $s$
cells in each row strictly above row $\posi_j$ if and only if
$d_j=1$. If $d_j=0$, we do not delete cells above row $\posi_j$. This
is because the observation in \cref{delete} on removing non-zero remainders says that the statistic $c_s$ does not change when deleting the last $\rho_j$ cells in row $\posi_j$ if and only if $d_j=0$.
\end{proof}

It turns out that the generating function in \Cref{gen} can be simplified.

\begin{theorem}
\label{theogen}
Let $\rem$ be a vector of integers between $1$ and~$s-1$ of length~$m$. The generating function
with respect to the
weight $R^{r_s(\lambda)} C^{c_s(\lambda)} q^{|\lambda|}$
of partitions~$\lambda$ with remainder sequence~$\rem$ is
$$
q^{|\rem|} Q^{-\wmaj(\rem)}  \sum_{i \ge m} Q^{\binom{m}{2}+i-m}
 \begin{bmatrix} i-1 \\ m-1 \end{bmatrix}_Q
\left( R^{i-m} +  \sum_{k \ge 1}  \frac{C Q^{k}}{(C Q; Q)_k}
R^{\max(i-m,k-m)} \right)
$$
where $Q=q^s$.
\end{theorem}

The theorem follows from \cref{gen},
the observation that,
for fixed $\posi_m$, we have
\begin{equation}
\label{qbin}
\sum_{1 \le \posi_1 < \posi_2 < \dots < \posi_m} Q^{|\pos|-m}
= Q^{\binom{m}{2}+\posi_m-m}  \begin{bmatrix} \posi_m-1 \\ m-1 \end{bmatrix}_Q,
\end{equation}
and from \cref{maj} below.
\Cref{qbin} holds
since $\left[\begin{smallmatrix} n+m \\ m\end{smallmatrix} \right]_q$
is the generating function $\sum_\lambda q^{|\lambda|}$ of
partitions~$\lambda$ of length at most~$m$ and parts no greater than
$n$, and since
$$
\sum_{1 \le \posi_1 < \posi_2 < \dots < \posi_m} Q^{|\pos|-m} =
Q^{\posi_m-m+1+2+\dots+m-1} \sum_{0 \le \posi^-_1 \le \posi^-_2 \le  \dots \le  \posi^-_{m-1} \le \posi_m-m} Q^{\posi^-_1+\dots+\posi^-_{m-1}},
$$
by the transformation $\posi^-_k=\posi_k-k$.

\begin{lemma}
\label{maj}
Let $\rem$ be a vector of integers between $1$ and~$s-1$ of length~$m$. Then, for fixed $\posi_m$, we have
$$
\sum_{1 \le \posi_1 < \posi_2 < \dots < \posi_m} Q^{{\mathbf d}(\rem,\pos) \cdot ({\pos- \mathbf 1})} =
Q^{-\wmaj(\rem)} \sum_{1 \le \posi_1 < \posi_2 < \dots < \posi_m} Q^{|\pos|-m},
$$
where $\pos=(\posi_1,\dots,\posi_m)$.
\end{lemma}

\begin{proof}
We need the following generalization of the weak major index: for $k$ with $1 \le k < m$,
we define
$$
\wmaj_k(\rem)= \underset{j\le k}{\sum_{j:\rho_j \ge \rho_{j+1}}} j.
$$
This simply is the weak major index of the tuple $\rem$ cut off after
the {$(k+1)$-st entry.} Note that $\wmaj_{m-1}=\wmaj$ for sequences of
length~$m$. 

The proof is by induction on~$m$. For the start of the
induction we note that for $m=1$ the statement is obvious.

In the following arguments, the reader should always keep in mind
that, by assumption, $\posi_m$ is fixed throughout.
In particular, if we write a sum 
$\sum_{1 \le \posi_1 < \posi_2 < \dots < \posi_m} \dots$, then the
sum runs over the $\posi_i$ with $i<m$, while $\posi_m$ is fixed. 

Now, by the induction hypothesis, we may assume
\begin{multline*}
\sum_{1 \le \posi_1 < \posi_2 < \dots < \posi_m} Q^{{\mathbf d}(\rem, \pos) \cdot ({\pos- \mathbf 1})}
=\sum_{\posi_{m-1}<\posi_{m}} Q^{d_m(\gamma_m -1)} \sum_{1 \le \posi_1 < \posi_2 < \dots < \posi_{m-1}} Q^{\sum_{j=1}^{m-1} d_j(\posi_j-1)}\\ 
=Q^{-\wmaj_{m-2}(\rem)}
\sum_{1 \le \posi_1 < \posi_2 < \dots < \posi_m} Q^{\sum_{j=1}^{m-1} (\posi_j-1) + d_{m} (\posi_{m}-1)}.
\end{multline*}
If $\rho_{m-1} < \rho_m$, then $d_m=1$ and
$\wmaj(\rem)=\wmaj_{m-2}(\rem)$, and the assertion follows in this
case. If, on the other hand, we have $\rho_{m-1} \ge \rho_m$, then
$\wmaj(\rem)=\wmaj_{m-2}(\rem)+m-1$, and, by the definition of $d_m$, we
have
\begin{multline}
\label{left}
\sum_{1 \le \posi_1 < \posi_2 < \dots < \posi_{m-1} < \posi_m} Q^{\sum_{j=1}^{m-1} (\posi_j-1) + d_{m} (\posi_{m}-1)} \\
= \sum_{1 \le \posi_1 < \posi_2 < \dots < \posi_{m-1} < \posi_m-1} Q^{|\pos|-m}
+ \sum_{ 1 \le \posi_1 < \posi_2 < \dots < \posi_{m-2} < \posi_m-1} Q^{\sum_{j=1}^{m-2} (\posi_j-1)+\posi_{m}-2},
\end{multline}
where the first sum of the right-hand side corresponds to the case $\posi_m > \posi_{m-1}+1$ and the second sum corresponds to the case $\posi_m = \posi_{m-1}+1$.

We need to show that this is equal to
$$
Q^{-m+1} \sum_{1 \le \posi_1 < \posi_2 < \dots < \posi_m} Q^{|\pos|-m}.
$$
We provide a combinatorial proof.
First note that the first term in the second line of \eqref{left} can be transformed as follows:
$$
\sum_{1 \le \posi_1 < \posi_2 < \dots < \posi_{m-1} < \posi_m-1} Q^{|\pos|-m} =
Q^{-m+1} \sum_{2 \le \posi_1 < \posi_2 < \dots < \posi_{m-1} < \posi_m} Q^{|\pos|-m}.
$$
Here we have used the transformation $\posi_i \to \posi_i-1$ for $i \in \{1,2,\dots,m-1\}$. The second term
in the second line of \eqref{left} is
$$
\sum_{1 \le \posi_1 < \posi_2 < \dots < \posi_{m-2} < \posi_m-1} Q^{\sum_{j=1}^{m-2} (\posi_j-1)+\posi_{m}-2} =
Q^{-m+1} \sum_{1=\posi_1 < \posi_2 < \dots < \posi_{m-1} < \posi_m} Q^{|\pos|-m},
$$
where we have used the transformation $\posi_i \to \posi_{i+1}-1$ for $i \in \{1,2,\dots,m-2\}$ and have set $\posi_1=1$.
This completes the proof.
\end{proof}

We will now use the combinatorial proof of the previous lemma to provide the missing piece of our bijection.
More concretely, the combinatorial argument allows us to reduce everything to the essence of \Cref{strict}.

We extend the notion of the remainder diagram to the case of arbitrary
remainder sequences as follows. To explain it, consider the example
for $s=3$ in \cref{fig:7}.

\begin{figure}[ht]
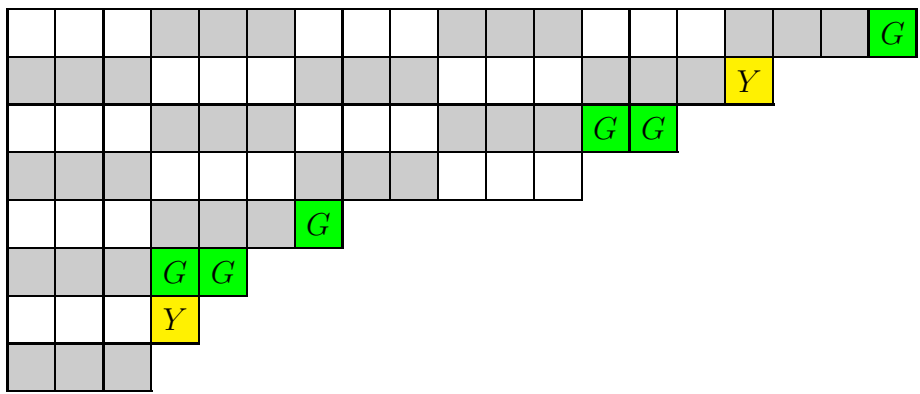

$$
\begin{ytableau}
\oc & \oc & \oc &
\ec & \ec & \ec & \oc & \oc & \oc & \ec & \ec & \ec & \oc & \oc & \oc & \ec & \ec & \ec & \gc \\
\ec & \ec & \ec & \oc & \oc & \oc & \ec & \ec & \ec & \oc & \oc & \oc & \ec & \ec & \ec & \yc \\
\oc & \oc & \oc & \ec & \ec & \ec & \oc & \oc & \oc & \ec & \ec & \ec & \gc & \gc \\
\ec & \ec & \ec & \oc & \oc & \oc & \ec & \ec & \ec & \oc & \oc & \oc  \\
\oc & \oc & \oc & \ec & \ec & \ec & \gc  \\
\ec & \ec & \ec & \gc & \gc  \\
\oc & \oc & \oc & \yc  \\
\ec & \ec & \ec
\end{ytableau}
$$
\caption{Green and yellow remainders in a partition}
\label{fig:7}
\end{figure}
Consider the \hbox{$i$-th} remainder from the bottom~(!), $i\geq 1$.  This
remainder is marked green if
$c_s(\Delta^i_s \lambda) = c_s(\Delta^{i-1}_s \lambda)+1$, and  it is
marked yellow if
$c_s(\Delta^i_s \lambda) = c_s(\Delta^{i-1}_s \lambda)$ (cf. 
\cref{delete}). The only
exception from this rule is a non-zero remainder in the top row,
which is always marked green; see \cref{fig:7}.

For a partition $\lambda$ let, as before, $\gout_s(\lambda)$ be the
set of green cells that correspond to the green remainders, and let $\yout_s(\lambda)$ be the set of yellow
cells that correspond to the yellow remainders.
Yellow cells are also located outside of the $s$-reduction;
in their row, they are adjacent to the final cell  of the $s$-reduction, 
however, they need not be outer corners of the $s$-reduction. In the following, we sometimes refer to the green and the yellow cells as 
the coloured cells.

The Ferrers diagram of $\nu=\lab{s}$ together
with the yellow and green cells is the \Dfn{(extended) remainder
  diagram $\nu^+_s(\lambda)$ for~$\lambda$}.  For the example above,
  the remainder diagram is shown in \cref{fig:8}.

\begin{figure}[ht]
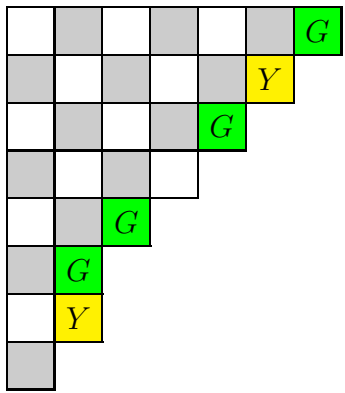

$$
\begin{ytableau}
\oc & \ec & \oc & \ec & \oc & \ec &  \gc \\
\ec & \oc & \ec & \oc & \ec &  \yc \\
\oc & \ec & \oc & \ec &  \gc \\
\ec & \oc & \ec & \oc  \\
\oc & \ec & \gc \\
\ec & \gc \\
\oc & \yc \\
\ec
\end{ytableau}
$$
\caption{The (extended) remainder diagram for the partition in
\cref{fig:7}}
\label{fig:8}
\end{figure}

More generally, an \Dfn{(extended) remainder diagram} $\nu^+$ is a partition
$\nu$ together with a collection of green cells $\gout(\nu^+)$ and a
collection of yellow cells $\yout(\nu^+)$, none of them in $\nu$,
provided the following three conditions are met:
\begin{itemize}
\item Green cells are outer corners of $\nu$.
\item Yellow cells are located at the end of a (possibly empty) row of $\nu$.
\item The cell at the end of the row preceding a row with a yellow
  cell is always a coloured cell of $\nu$ (cf. the first case in 
  \cref{delete}).  In particular, a coloured cell in
  the top row must be green.
\end{itemize}

A remainder diagram with coloured cells in rows $\pos$ is \Dfn{compatible} with a vector $\rem$ of integers between $1$ and~$s-1$ provided that for any weak descent $\rho_{k-1}\geq\rho_{k}$
of $\rem$ the coloured cell in row $\posi_k$ is yellow if and only if
$\posi_{k-1} = \posi_k - 1$.

We can now express the statistics $r_s$ and $c_s$ in terms of the
(extended) remainder diagram, thus generalizing \cref{lem:increasing}.
\begin{lemma}\label{lem:statistics-remainder-diagram}
  Let $\lambda$ be a partition with remainder diagram
  $\nu^+_s(\lambda)$.  Then
  \begin{align*}
    r_s(\lambda) &= \# \text{ of rows of } \nu^+_s(\lambda)-|\gout_s(\lambda)|-|\yout_s(\lambda)|, \\
    c_s(\lambda) &= \# \text{ of columns of } \nu^+_s(\lambda)-|\gout_s(\lambda)|.
  \end{align*}
\end{lemma}
\begin{proof}
The proof is analogous to that of \Cref{lem:increasing}. Note that coloured cells correspond to the parts of $\lambda$ that are not divisible by $s$. 
Furthermore, by \cref{delete}, yellow cells identify all but the first row where the statistic~$c_s$ remains unchanged when cells are removed from that row through successive applications of $\Delta_s$.
\end{proof}

Next we describe a bijection between remainder diagrams compatible
with a given remainder sequence and remainder diagrams without yellow
cells.  To state it precisely, inspired by
\Cref{lem:statistics-remainder-diagram} we define for any remainder
diagram $\nu^+$ the two statistics
\begin{align*}
  r(\nu^+) &= \# \text{ of rows of } \nu^+ -
             |\gout(\nu^+)|-|\yout(\nu^+)|\quad\text{and}\\
  c(\nu^+) &= \# \text{ of columns of } \nu^+ - |\gout(\nu^+)|.
\end{align*}
The following construction is a translation of the combinatorial
proof of \cref{maj}.
\begin{construction}[\sc Reduction to remainder diagrams without yellow
  cells]
  \label{reduce}
  Let $\lambda$ be a partition with remainder sequence
  $\rem=\rem_s(\lambda)$ and remainder diagram $\nu^+_s(\lambda)$.

\begin{enumerate}[label=\textnormal{({\arabic*})},resume]
\item\label{S1} Initialization: We let $k:=1$,
  $\nu^+:=\nu_s^+(\lambda)$,
  and $\nu:=\lab{s}$.
\item\label{S2}If $k$ equals the length of $\rem$ then go to~\ref{S5}.
If not and if there is a weak descent of $\rem$ at $k$, i.e., if
$\rho_{k} \ge \rho_{k+1}$, then go to~\ref{S3}. Otherwise increase $k$
  by~$1$ and repeat \ref{S2} with this new value of~$k$.
\item By construction, all coloured cells strictly
  above row $\posi_{k+1}$ in the diagram~$\nu^+$ are already green and
  thus outer corners of~$\nu$.

\begin{itemize}
\item[\em(3A)]\label{S3}
  If the coloured cell in row~$\posi_{k+1}$
  is green, then the next green cell above is
  not in row $\posi_{k+1}-1$, i.e., $\posi_{k} < \posi_{k+1}-1$
(cf.\ the definition of the (extended) remainder diagram and
\cref{delete}).  We add 
  the outer corners of~$\nu$ in rows $\posi_1,\dots,\posi_{k}$ to~$\nu$
  and add for each of them a green
  cell to $\nu^+$ in the row below.

\item[\em(3B)]\label{S4} If the coloured cell in row~$\posi_{k+1}$ is yellow, then
  we delete this coloured cell from~$\nu^+$. 
  In this case, the next coloured cell above is in row $\posi_{k+1}-1$,
  i.e., $\posi_{k}= \posi_{k+1}-1$, and all coloured cells above row~$\posi_{k+1}$ are outer corners. 
  We add the (coloured) outer corners in rows
  $\posi_1,\dots,\posi_{k}$ to~$\nu$
  and add for each of them a green cell to~$\nu^+$ in the row
  below\footnote{Note that this has the effect that the former yellow
    cell in row $\posi_{k+1}$ is replaced by a green cell.}.
  Finally, we add a green cell to~$\nu^+$ in the first row.
\end{itemize}
Increase $k$ by $1$ and go to \ref{S2}.
\item\label{S5} The output of the algorithm is the remainder
  diagram~$\nu^+$ with interior~$\nu$.
\end{enumerate}
\end{construction} 

We illustrate this construction with the help of the example in
\cref{fig:7} with remainder diagram in \cref{fig:8}. 
In this case, the row position sequence is $(1,2,3,5,6,7)$ and
the remainder sequence is $(1,1,2,1,2,1)$. Hence
we have weak descents of the remainder sequence 
at $k=1$, $3$, and $5$. The sequence of pairs $(\nu^+,\nu)$
we obtain when applying the algorithm of \Cref{reduce} is 
shown in \cref{fig:9}. There, the white and shaded cells form the
partitions~$\nu$, while the complete diagrams --- including the green
and yellow cells --- form the partitions~$\nu^+$.
Note that the final remainder diagram is not compatible
with the original remainder sequence.
\begin{figure}[ht]
\ytableausetup{smalltableaux}
$$
\begin{ytableau}
\oc & \ec & \oc & \ec & \oc & \ec &  \gc \\
\ec & \oc & \ec & \oc & \ec &  \yc \\
\oc & \ec & \oc & \ec &  \gc \\
\ec & \oc & \ec & \oc  \\
\oc & \ec & \gc \\
\ec & \gc \\
\oc & \yc \\
\ec
\end{ytableau}
\quad \mapsto\quad 
\begin{ytableau}
\oc & \ec & \oc & \ec & \oc & \ec &  \oc &  \gc \\
\ec & \oc & \ec & \oc & \ec &  \gc \\
\oc & \ec & \oc & \ec &  \gc \\
\ec & \oc & \ec & \oc  \\
\oc & \ec & \gc \\
\ec & \gc \\
\oc & \yc \\
\ec
\end{ytableau}
\quad \mapsto\quad 
\begin{ytableau}
\oc & \ec & \oc & \ec & \oc & \ec &  \oc & \ec \\
\ec & \oc & \ec & \oc & \ec &  \oc & \gc \\
\oc & \ec & \oc & \ec &  \oc & \gc \\
\ec & \oc & \ec & \oc & \gc  \\
\oc & \ec & \gc \\
\ec & \gc \\
\oc & \yc \\
\ec
\end{ytableau}
\quad \mapsto\quad 
\begin{ytableau}
\oc & \ec & \oc & \ec & \oc & \ec &  \oc & \ec & \gc \\
\ec & \oc & \ec & \oc & \ec &  \oc & \ec \\
\oc & \ec & \oc & \ec &  \oc & \ec & \gc\\
\ec & \oc & \ec & \oc & \ec & \gc  \\
\oc & \ec & \oc & \gc \\
\ec & \oc & \gc \\
\oc & \gc \\
\ec
\end{ytableau}
$$
\ytableausetup{nosmalltableaux}
\caption{Application of \cref{reduce}}
\label{fig:9}
\end{figure}

The following lemma confirms that \cref{reduce} has all the required
properties such that it indeed achieves the desired reduction
to the case of remainder diagrams without yellow cells.

\begin{lemma}\label{lem:constr3}  
For any positive integer~$s$, \cref{reduce} yields a bijection between partitions $\lambda$ and remainder diagrams~$\nu^+$ without yellow cells, satisfying $r(\nu^+) = r_s(\lambda)$ and $c(\nu^+) = c_s(\lambda)$, and whose interior has $\wmaj(\rem_s(\lambda))$ more cells than $\lab{s}$.
\end{lemma}

\begin{proof} 
Let $\lambda$ be a partition with remainder sequence
	$\rem=\rem_s(\lambda)$ and remainder diagram $\nu^+_s(\lambda)$. To see that the interior of $\nu$ has increased by $\wmaj(\rem)$ after applying \cref{reduce}, note
that in both~(3A) and~(3B) we add $k$ cells to the interior of the remainder diagram, which is precisely the contribution of the weak
descent at position~$k$ to the weak major index of $\rem$.  

To see that $r(\nu^+) = r_s(\lambda)$ and
  $c(\nu^+) = c_s(\lambda)$, note that the total
number of cells which are either green or yellow and also the number
of rows do not change.  In a step~(3A),
the number of columns
does not change either.  In a step~(3B),
the number of columns
increases by one, and so does the number of green cells.

Each step of the construction is invertible, since we can determine
from the image in which of steps~(3A) or~(3B) we were: there is a green
cell in the first row of the image if and only if the coloured
cell in row $\posi_{k+1}$ in the preimage is yellow.
\end{proof} 

In the example in \cref{fig:9}, applying the inverse of \Cref{reduce} to the
conjugate of the last diagram with respect to the remainder
sequence $(1,1,2,1,2,1)$, we obtain the sequence of diagrams in \cref{fig:20}.

\begin{figure}[ht]
\ytableausetup{smalltableaux}
$$
\begin{ytableau}
\oc & \ec & \oc & \ec & \oc & \ec &  \oc & \ec  \\
\ec & \oc & \ec & \oc & \ec &  \oc & \gc \\
\oc & \ec & \oc & \ec &  \oc & \gc  \\
\ec & \oc & \ec & \oc & \gc    \\
\oc & \ec & \oc & \ec \\
\ec & \oc & \ec & \gc  \\
\oc & \ec & \gc  \\
\ec \\
\gc 
\end{ytableau}
\quad \mapsto\quad 
\begin{ytableau}
\oc & \ec & \oc & \ec & \oc & \ec &  \oc & \gc  \\
\ec & \oc & \ec & \oc & \ec &  \gc \\
\oc & \ec & \oc & \ec &  \gc  \\
\ec & \oc & \ec & \oc     \\
\oc & \ec & \oc & \gc \\
\ec & \oc & \gc   \\
\oc & \ec  \\
\ec \\
\gc 
 \end{ytableau}
\quad \mapsto\quad 
\begin{ytableau}
\oc & \ec & \oc & \ec & \oc & \ec &  \gc  \\
\ec & \oc & \ec & \oc & \gc \\
\oc & \ec & \oc & \ec   \\
\ec & \oc & \ec & \gc     \\
\oc & \ec & \oc & \yc \\
\ec & \oc & \gc   \\
\oc & \ec  \\
\ec \\
\gc 
 \end{ytableau}
\quad  \mapsto\quad 
\begin{ytableau}
\oc & \ec & \oc & \ec & \oc & \gc  \\
\ec & \oc & \ec & \oc & \yc \\
\oc & \ec & \oc & \ec   \\
\ec & \oc & \ec & \gc     \\
\oc & \ec & \oc & \yc \\
\ec & \oc & \gc   \\
\oc & \ec  \\
\ec \\
\gc 
 \end{ytableau}
$$
\ytableausetup{nosmalltableaux}
\caption{Application of the inverse of  \cref{reduce}}
\label{fig:20}
\end{figure}

\medskip
We can now put together the preceding constructions to obtain a bijection for the general case.
\begin{construction} \label{general}
  Let $\lambda$ be a partition and $\rem=\rem_s(\lambda)$ its remainder sequence.
  \begin{itemize}
  \item We apply \Cref{reduce} to $\nu^+_s(\lambda)$ with respect to $\rem$ to obtain a remainder diagram $\mu^+$ without yellow cells.
  \item We apply the inverse of \Cref{reduce} to $[\mu^+]'$ with respect to $\rem$ to obtain a remainder diagram $\kappa^+$.
  \item We transform $\kappa^+$ into a partition by applying the $s$-blow-up to the interior of the diagram and replacing the coloured cells by the remainders of the original partition to obtain $\kappa^+\leftarrow_s\rem$ (compare with \cref{strict}).
  \end{itemize}
\end{construction}

Note that the resulting partition is again compatible with $\rem$ by
construction.

\begin{example} 
We apply \cref{general} to the partition in \cref{fig:7}. Its remainder diagram is displayed in \cref{fig:8}. The application of \Cref{reduce} to the remainder diagram is performed in \cref{fig:9}. Let 
$\mu^+$ denote the resulting remainder diagram.
The application of the inverse of \Cref{reduce} to  $[\mu^+]'$ is performed in \cref{fig:20}. After applying the $s$-blow-up to the interior of the diagram and putting the remainders back, we obtain the partition in \cref{fig:30}.

\begin{figure}[ht]
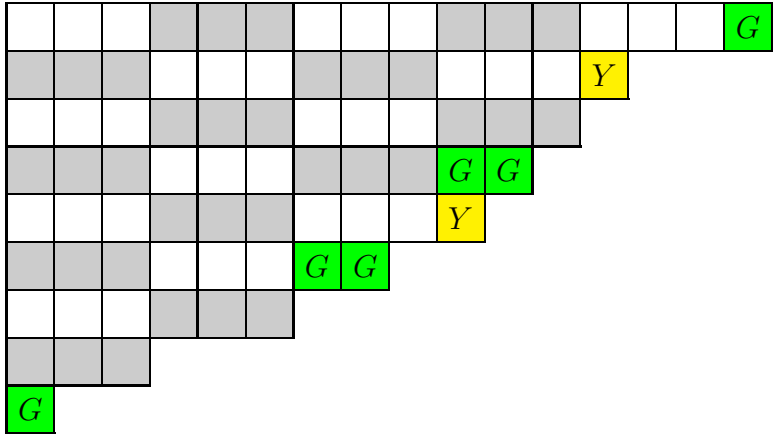

$$
\begin{ytableau}
\oc & \oc & \oc &
\ec & \ec & \ec & \oc & \oc & \oc & \ec & \ec & \ec & \oc & \oc & \oc & \gc \\
\ec & \ec & \ec & \oc & \oc & \oc & \ec & \ec & \ec & \oc & \oc & \oc & \yc \\
\oc & \oc & \oc & \ec & \ec & \ec & \oc & \oc & \oc & \ec & \ec & \ec  \\
\ec & \ec & \ec & \oc & \oc & \oc & \ec & \ec & \ec & \gc & \gc   \\
\oc & \oc & \oc & \ec & \ec & \ec & \oc & \oc & \oc & \yc  \\
\ec & \ec & \ec & \oc & \oc & \oc &  \gc & \gc  \\
\oc & \oc & \oc & \ec & \ec & \ec   \\
\ec & \ec & \ec \\
\gc 
\end{ytableau}
$$
\caption{Partition obtained after applying the bijection summarized in \cref{general} to the partition in \cref{fig:7}} 
\label{fig:30}
\end{figure}
\end{example}

\begin{remark}
The partitions of 37 in \cref{ex:1} appear in the order as ``dictated"
by the involution in \cref{general}. 
To be precise, if one applies \cref{general} to 
the partitions in \cref{fig:2} then the output partitions are the
ones in \cref{fig:3}, in the given order.
\end{remark}

\section{Proof of \cref{coefficient}}
\label{sec:5}
By \cref{main}, the generating functions in 
\cref{zero,,inc,,gen,,theogen} are all symmetric in
$R$ and $C$, however this is not visible from the formulas. We transform the formula in \cref{theogen} and extract the coefficient
of $R^r C^c$ to obtain a form where the symmetry in $R$ and $C$ is
obvious. The result is stated in \cref{coefficient} (with the
symmetric rewriting of the formula given in
\cref{rem:1a},
and the
content of this 
section is its proof. We start with a proof by computation and provide a combinatorial proof afterwards.

\begin{proof}[First proof]
First note that we can extend the sum over $i$ in the formula in
\cref{theogen} over all $i \ge 0$ since $\left[ \begin{smallmatrix}i-1
\\ m-1\end{smallmatrix} \right]_Q=0$ if $0 \le i < m$. We neglect the
prefactor $q^{|\rem|} Q^{-\wmaj(\rem)+\binom{m}{2}}$ in the formula
since it is independent of $R$ 
and $C$, and we start by decomposing the sum over~$k$ in the formula in \cref{theogen} to get rid of the maximum as
\begin{multline}\label{EquationMultipleLines}
     \sum_{i \ge 0} Q^{i-m}
 \begin{bmatrix} i-1 \\ m-1 \end{bmatrix}_Q
\left( R^{i-m} + \sum_{k \ge 1}  \frac{C Q^{k}}{(C Q; Q)_k}
R^{\max(i-m,k-m)} \right)\\
=
\sum_{i \ge 0} (RQ)^{i-m}
 \begin{bmatrix} i-1 \\ m-1 \end{bmatrix}_Q
 + \sum_{k \ge 1}  \frac{C Q^{k}}{(C Q; Q)_k}
\sum_{i > k} (RQ)^{i-m}
 \begin{bmatrix} i-1 \\ m-1 \end{bmatrix}_Q \\
 + \sum_{k \ge 1}  \frac{C Q^{k}R^{k-m}}{(C Q; Q)_k}\sum_{i =0}^k Q^{i-m}
 \begin{bmatrix} i-1 \\ m-1 \end{bmatrix}_Q.
\end{multline}
We rewrite the first term as
$$
 \sum_{i \ge m} (RQ)^{i-m}
 \begin{bmatrix} i-1 \\ m-1 \end{bmatrix}_Q= \sum_{i \ge m} (RQ)^{i-m} \frac{(Q^{m};Q)_{i-m}}{(Q;Q)_{i-m}}
 =\sum_{i\geq 0}(RQ)^i\frac{(Q^  {m};Q)_{i}}{(Q;Q)_{i}}.
 $$
By the $Q$-binomial theorem (cf.\
\cite[Eq.~(1.3.2); Appendix (II.3)]{MR2128719}) the last sum evaluates
to
$$
\frac {(RQ^{m+1};Q)_\infty} {(RQ;Q)_\infty}
=\frac {1} {(RQ;Q)_m}.
$$ 
Thus, we arrive at the expression
\begin{multline}\label{GFSplitIntoPartsWithoutMax}
 \frac{1}{(RQ;Q)_m}
 + \sum_{k \ge 1}  \frac{C Q^{k}}{\prod_{i=1}^{k}(1-C Q^{i})}
\sum_{i > k} (RQ)^{i-m}
 \begin{bmatrix} i-1 \\ m-1 \end{bmatrix}_Q \\
 + \sum_{k \ge 1}  \frac{C Q^{k}R^{k-m}}{\prod_{i=1}^{k}(1-C Q^{i})}\sum_{i =0}^k Q^{i-m}
 \begin{bmatrix} i-1 \\ m-1 \end{bmatrix}_Q.
\end{multline}
Next we extract the coefficient of $R^rC^c$. In order to do so, we will make use of the simple expansion
$$\frac{1}{(z;Q)_k}=\sum_{l\geq 0}\begin{bmatrix} l+k-1 \\ l \end{bmatrix}_Q z^l.$$
We start with the case $c=0$. Making use of the expansion above, we see that the coefficient of $R^r$ in \eqref{GFSplitIntoPartsWithoutMax} is
\begin{equation}\label{CoefficientR}
    Q^r \begin{bmatrix} r+m-1 \\ r\end{bmatrix}_Q.
\end{equation}
If we let $c\geq 1$, making again use of the expansion above, we see that the coefficient of $R^rC^c$ in \eqref{GFSplitIntoPartsWithoutMax} equals
\begin{equation}\label{CoefficientofRC}
Q^r\begin{bmatrix} r+m-1 \\ m-1 \end{bmatrix}_Q\sum_{k=1}^{r+m-1}Q^{k+c-1}\begin{bmatrix} c+k-2 \\ c-1 \end{bmatrix}_Q
    +Q^{r+m+c-1}\begin{bmatrix} r+c+m-2 \\ c-1 \end{bmatrix}_Q \sum_{i=m}^{r+m}Q^{i-m}\begin{bmatrix} i-1 \\ m-1 \end{bmatrix}_Q.
\end{equation}
Using the simple summation
\begin{equation}\label{eq:simplesum}
\sum_{k=1}^NQ^{k-1}\begin{bmatrix} M+k-1 \\ M \end{bmatrix}_Q=\begin{bmatrix} M+N \\ M+1 \end{bmatrix}_Q,
\end{equation}
the first expression in \eqref{CoefficientofRC} can be evaluated to
$$
    Q^r\begin{bmatrix} r+m-1 \\ m-1 \end{bmatrix}_Q\sum_{k=1}^{r+m-1}Q^{k+c-1}\begin{bmatrix} c+k-2 \\ c-1 \end{bmatrix}_Q
    =Q^{r+c}\begin{bmatrix} r+m-1 \\ m-1 \end{bmatrix}_Q\begin{bmatrix} r+c+m-2 \\ c\end{bmatrix}_Q.
$$
Using the same summation~\eqref{eq:simplesum}, we also
see that
\[
\sum_{i=m}^{r+m}Q^{i-m}\begin{bmatrix} i-1 \\ m-1 \end{bmatrix}_Q = \begin{bmatrix} r+m \\ m \end{bmatrix}_Q.
\]
Putting everything back together, we see that \eqref{CoefficientofRC} equals
\begin{equation*}
 Q^{r+c}\begin{bmatrix} r+m-1 \\ m-1 \end{bmatrix}_Q\begin{bmatrix} r+c+m-2 \\ c\end{bmatrix}_Q
    +Q^{r+c+m-1}\begin{bmatrix} r+c+m-2 \\ c-1 \end{bmatrix}_Q \begin{bmatrix} r+m \\ m \end{bmatrix}_Q,
\end{equation*}
which, aside from the neglected prefactor $q^{|\rem|}
Q^{-\wmaj(\rem)+\binom{m}{2}}$, is exactly the expression in
\cref{coefficient}. 
\end{proof}

\begin{proof}[Second proof]

By \cref{lem:constr3}, it suffices to show that the generating function with respect to the weight $Q^{\# \text{ of interior cells}}$ of remainder diagrams without yellow cells, where the number of rows is $r+m$ and the number of columns is $c+m$ (including rows and columns of green cells), is given by
\begin{multline}\label{EquationForCombInterpretation}
	Q^{0+1+\dots+m-2}Q^{r+c+m-1}\begin{bmatrix} r+m-1 \\ m-1 \end{bmatrix}_Q\begin{bmatrix} r+c+m-2 \\ c\end{bmatrix}_Q \\
	+Q^{0+1+\dots+m-1}Q^{r+m+c-1}\begin{bmatrix} r+c+m-2 \\ c-1 \end{bmatrix}_Q \begin{bmatrix} r+m \\ m \end{bmatrix}_Q.
\end{multline}

We distinguish between two cases.

\medskip
\noindent
{\sc Case~1}: {\sc the bottom row of the remainder diagram contains an
interior cell.} 
 We claim that this case is covered by the second summand in
    \eqref{EquationForCombInterpretation}. To see this, consider
    $\lab{s}$, and let us decompose it as follows. First we cut out the
    columns of the green cells in~$\lab{s}$.  This gives us
    $m$ columns of distinct lengths where the largest column has at most
    $r+m-1$ boxes, the smallest one being allowed to be empty,
    since  $\lab{s}$ has $r+m$ rows and no green cell is
    below the last row of~$\lab{s}$. 
     
     We illustrate this with the example in \cref{fig:5}, with the remainder diagram given in \cref{fig:6}. 
   Note that we have   $m=3,c=2,r=3$ in this case.
     In \cref{eq:1},
    the $m$ columns of distinct lengths appear as the first shape on the right-hand
    side. The corresponding generating function is
    \textcolor{leftcell}{{$Q^{0+1+\dots+m-1}\left[\begin{smallmatrix}
    r+m \\ m\end{smallmatrix} \right]_Q$}}, where here and in the following 
    the colours of the expressions hint at the corresponding parts of the Ferrers diagrams in the figures.
    
    Next we cut off the outer
    frame of the remaining partition, that is, all boxes of the first
    row and first column. This gives us $r+m-1+c$ boxes; compare with
    the second shape on the right-hand side of \cref{eq:1}.  These boxes are
    taken into account by \textcolor{middlecell}{$Q^{r+c+m-1}$}. What
    remains is a partition with at most $c-1$ columns of size at most
    $r+m-1$, as in the final shape in \cref{eq:1}.  The corresponding 
generating
    function is given by the remaining factor
    \textcolor{rightcell}{{$\left[\begin{smallmatrix} r+c+m-2 \\
    c-1\end{smallmatrix} \right]_Q$}}. 

\begin{figure}[ht]
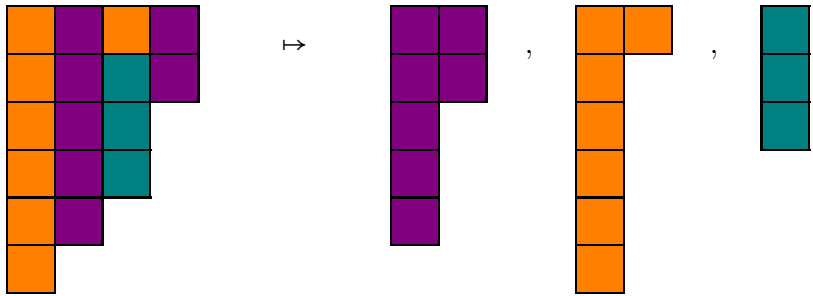

$$
    \begin{ytableau}
    \mc & \lc & \mc & \lc \\
    \mc & \lc & \rc & \lc \\
    \mc & \lc & \rc \\
    \mc & \lc & \rc \\
    \mc & \lc \\
    \mc \\
    \end{ytableau}
    \hspace{1cm}
    \mapsto
    \hspace{1cm}
    \begin{ytableau}
    \lc & \lc \\
    \lc  & \lc \\
    \lc \\
    \lc\\
    \lc \\
    \end{ytableau}
    \hspace{0.5cm}
    ,
    \hspace{0.5cm}
    \begin{ytableau}
    \mc & \mc \\
    \mc  \\
    \mc  \\
    \mc  \\
    \mc  \\
    \mc \\
    \end{ytableau}
    \hspace{0.5cm}
    ,
    \hspace{0.5cm}
    \begin{ytableau}
     \rc\\
     \rc \\
     \rc \\
    \end{ytableau}  
$$
\caption{The decomposition of the interior of the remainder diagram of \cref{fig:5} described in Case~1 of the second proof of \cref{coefficient}}
\label{eq:1}
\end{figure}

\medskip
\noindent
{\sc Case~2}: {\sc the bottom row of the remainder diagram consists
only of a green cell.} 
 There are $m-1$ green cells that are in the last row of
the $s$-reduced diagram or above. We claim that this case is covered
by the first summand in \eqref{EquationForCombInterpretation}. The
argument is analogous to the one above. Again we consider the
$s$-reduced diagram and decompose it as follows. We start by cutting
out the $m-1$ columns of the green cells different from the
bottommost green cell in~$\lab{s}$. This gives us $m-1$
columns of different lengths, where the largest has length at most
$r+m-2$, since $\lambda$ has $r+m-1$ rows and we only consider the
green cells different from the bottommost green cell. 

 We illustrate this with the example in \cref{fig:12}, with the remainder diagram given in \cref{fig:13}. 
 Note that we have  
 $m=2,c=3,r=5$ in this case.
In \cref{eq:2}, these $m-1$ columns of different lengths appear as the
first shape on the right-hand side. Here the generating function is
\textcolor{leftcell}{$Q^{0+1+\dots+m-2}\left[\begin{smallmatrix}
r+m-1 \\ m-1\end{smallmatrix} \right]_Q$}. From the remaining
partition we cut off the outer frame of $r+m-2+(c+1)$ boxes;
compare with the second shape on the right-hand side of \cref{eq:2}.
These boxes are taken into account by the factor
 \textcolor{middlecell}{$Q^{r+c+m-1}$}. What remains is a partition
with at most $c$ columns of length at most $r+m-2$, as in the final
shape in \cref{eq:2}.  The corresponding generating function is given
by the remaining factor of
\textcolor{rightcell}{$\left[\begin{smallmatrix}
r+c+m-2 \\ c \end{smallmatrix} \right]_Q$}. 
\begin{figure}[ht]
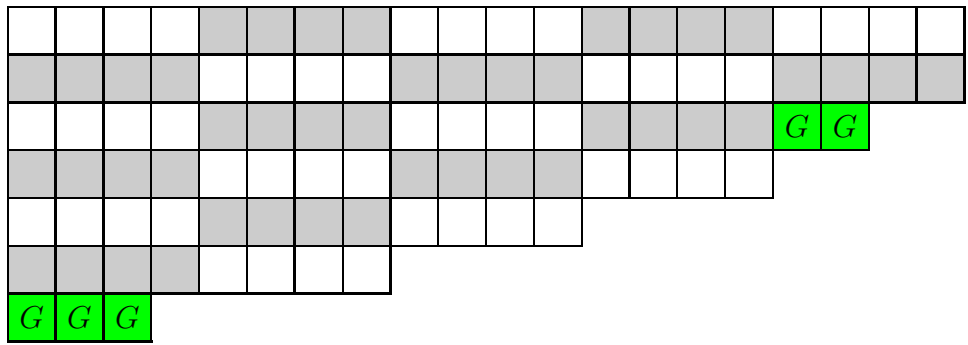

    $$
    \begin{ytableau}
    \oc & \oc & \oc & \oc &\ec & \ec & \ec & \ec & \oc & \oc & \oc & \oc & \ec & \ec & \ec & \ec & \oc & \oc & \oc & \oc \\
    \ec & \ec & \ec & \ec &\oc & \oc & \oc & \oc & \ec & \ec & \ec & \ec & \oc & \oc & \oc & \oc & \ec & \ec & \ec  & \ec \\
    \oc & \oc & \oc & \oc &\ec & \ec & \ec & \ec & \oc & \oc & \oc & \oc & \ec & \ec & \ec & \ec & \gc & \gc \\
    \ec & \ec & \ec & \ec &\oc & \oc & \oc & \oc & \ec & \ec & \ec & \ec & \oc & \oc & \oc & \oc \\
    \oc & \oc & \oc & \oc &\ec & \ec & \ec & \ec & \oc & \oc & \oc & \oc  \\
    \ec & \ec & \ec & \ec &\oc & \oc & \oc & \oc\\
    \gc & \gc & \gc \\
    \end{ytableau}
    $$
\caption{The example used to illustrate Case~2 of the second proof of  \cref{coefficient}}
\label{fig:12}
\end{figure}
\begin{figure}[ht]
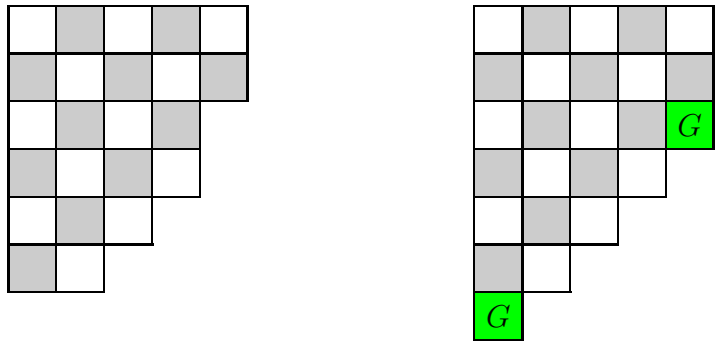

    $$
    \begin{ytableau}
    \oc &\ec & \oc & \ec & \oc \\
    \ec &\oc & \ec & \oc & \ec \\
    \oc &\ec & \oc & \ec \\
    \ec &\oc & \ec & \oc \\
    \oc & \ec & \oc \\
    \ec &\oc \\
    \end{ytableau}
    \hspace{3cm}
    \begin{ytableau}
    \oc &\ec & \oc & \ec & \oc\\
    \ec &\oc & \ec & \oc & \ec \\
    \oc &\ec & \oc & \ec & \gc \\
    \ec &\oc & \ec & \oc \\
    \oc &\ec & \oc \\
    \ec &\oc\\
    \gc \\
    \end{ytableau}
    $$
\caption{The remainder diagram of the partition in \cref{fig:12} (right) and its interior (left)}
\label{fig:13}
\end{figure}

\begin{figure}[ht]
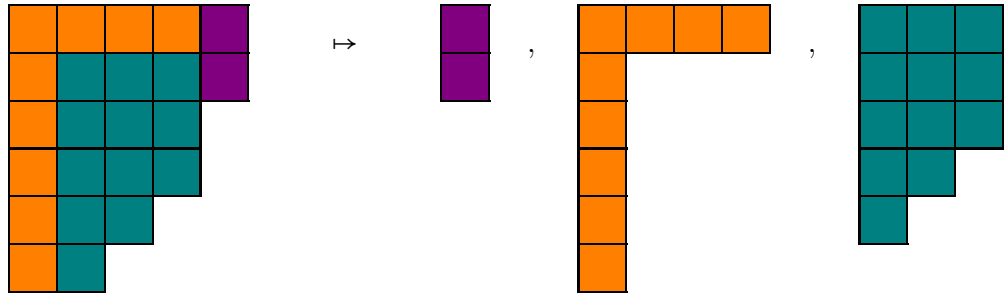

$$
    \begin{ytableau}
    \mc &\mc & \mc & \mc & \lc \\
    \mc &\rc & \rc & \rc & \lc \\
    \mc &\rc & \rc & \rc \\
    \mc &\rc & \rc & \rc \\
    \mc & \rc & \rc \\
    \mc &\rc \\
    \end{ytableau}
    \hspace{1cm}
    \mapsto
    \hspace{1cm}
    \begin{ytableau}
    \lc\\
    \lc\\
    \end{ytableau}
    \hspace{0.5cm}
    ,
    \hspace{0.5cm}
    \begin{ytableau}
    \mc &\mc & \mc & \mc\\
    \mc \\
    \mc \\
    \mc\\
    \mc \\
    \mc \\
    \end{ytableau}
     \hspace{0.5cm}
    ,
    \hspace{0.5cm}
    \begin{ytableau}
    \rc & \rc & \rc \\
    \rc & \rc & \rc \\
    \rc & \rc & \rc \\
    \rc & \rc \\
    \rc \\
    \end{ytableau}
$$
\caption{The composition of the interior of the remainder diagram of \cref{fig:12} described in Case~2 of the second proof of \cref{coefficient}}
\label{eq:2}
\end{figure}

\medskip

This completes the proof of the theorem.
\end{proof}

\begin{proof}[Proof of Corollary~\ref{rem:1a}]
To obtain the symmetric version of the generating function in \cref{coefficient}, we need to show that
\begin{multline}\label{eq:cor}
\begin{bmatrix} r+m-1 \\ m-1 \end{bmatrix}_Q\begin{bmatrix} r+c+m-2 \\ c\end{bmatrix}_Q
+Q^{m-1}\begin{bmatrix} r+m \\ m \end{bmatrix}_Q
\begin{bmatrix} r+c+m-2 \\ c-1 \end{bmatrix}_Q\\
= \frac {[r+c+m-1]_Q!} {[r]_Q!\,[c]_Q!\,[m-1]_Q!}
+Q^{m-1}\frac {[r+c+m-2]_Q!} {[r-1]_Q!\,[c-1]_Q!\,[m]_Q!}.
\end{multline}
The first summand on the left-hand side of \eqref{eq:cor} equals
\begin{multline*}
	\begin{bmatrix} r+m-1 \\ m-1 \end{bmatrix}_Q\begin{bmatrix} r+c+m-2 \\ c\end{bmatrix}_Q = \frac {[r+m-1]_Q!\,[r+c+m-2]_Q!} {[m-1]_Q!\,[r]_Q!\,[c]_Q!\,[r+m-2]_Q!}\\
	= \frac {[r+c+m-1]_Q!} {[m-1]_Q!\,[r]_Q!\,[c]_Q!} - \frac {[r+m-2]_Q!\,[r+c+m-1]_Q!-[r+m-1]_Q!\,[r+c+m-2]_Q!} {[m-1]_Q!\,[r]_Q!\,[c]_Q!\,[r+m-2]_Q!}\\
	= \frac {[r+c+m-1]_Q!} {[m-1]_Q!\,[r]_Q!\,[c]_Q!} - Q^{r+m-1}\frac {[r+c+m-2]_Q!} {[m-1]_Q!\,[r]_Q!\,[c-1]_Q!}.
\end{multline*}
Regarding the second summand, we analogously obtain
\begin{multline*}
	\begin{bmatrix} r+m \\ m \end{bmatrix}_Q
	\begin{bmatrix} r+c+m-2 \\ c-1 \end{bmatrix}_Q=\frac {[r+m]_Q!\,[r+c+m-2]_Q!} {[m]_Q!\,[r]_Q!\,[c-1]_Q!\,[r+m-1]_Q!}\\
	=\frac {[r+c+m-2]_Q!} {[r-1]_Q!\,[c-1]_Q!\,[m]_Q!} + Q^r \frac {[r+c+m-2]_Q!} {[r]_Q!\,[c-1]_Q!\,[m-1]_Q!}.
\end{multline*}
Putting it all together then yields the desired expression.
\end{proof}

\section*{Acknowledgement}
Four of the authors wish to express their gratitude to Deutsche Bahn
for delaying one of their trains on the way back from the 90th
S\'eminaire Lotharingien de Combinatoire, resulting in a nightly
stopover in a decent hotel in M\"unchen --- paid by Deutsche Bahn
--- during which this work was initiated.

\bibliographystyle{plain}
\bibliography{rc}
\end{document}